\documentclass[10pt,a4paper]{article}
\title{On $p$-adic $L$-functions for symplectic representations of $\GL(N)$ over number fields}
\author{Chris Williams}
\date{}

\usepackage[margin=1.2in]{geometry}

\newcommand{\s}{\setlength{\itemsep}{0pt}}

\usepackage{multicol}
\usepackage{amssymb}

\usepackage{fancyhdr,xcolor}
\newcommand{\cEv}{\mathrm{Ev}}
\newcommand{\sEv}{\mathscr{E}\hspace{-1pt}{\scriptstyle\mathscr{V}}}
\newcommand{\sEvs}{\mathscr{E}\hspace{-1pt}{\scriptscriptstyle\mathscr{V}}}

\newcommand{\UPS}{\theta} 

\newcommand{\usw}{\underline{\sw}}

\newcommand{\VH}{V^H_{\bj,-\underline{\sw}-\bj}}

\makeatletter 
\def\input@path{{../}} 
\makeatother

\usepackage[utf8]{inputenc}
\usepackage{amsmath, amsthm, amssymb, amsfonts}
\usepackage{enumerate, color}
\usepackage{hyperref}
\usepackage{calligra, mathrsfs}
\usepackage{bbm}
\usepackage{appendix}
\usepackage{url}

\usepackage{subfiles}

\usepackage{amscd}
\usepackage{lmodern,graphicx}
\usepackage[all]{xy} 
\xyoption{rotate}
\hypersetup{
	colorlinks,
	citecolor=black,
	filecolor=black,
	linkcolor=black,
	urlcolor=black
} 

\makeatletter
\DeclareFontFamily{OMX}{MnSymbolE}{}
\DeclareSymbolFont{MnLargeSymbols}{OMX}{MnSymbolE}{m}{n}
\SetSymbolFont{MnLargeSymbols}{bold}{OMX}{MnSymbolE}{b}{n}
\DeclareFontShape{OMX}{MnSymbolE}{m}{n}{
	<-6>  MnSymbolE5
	<6-7>  MnSymbolE6
	<7-8>  MnSymbolE7
	<8-9>  MnSymbolE8
	<9-10> MnSymbolE9
	<10-12> MnSymbolE10
	<12->   MnSymbolE12
}{}
\DeclareFontShape{OMX}{MnSymbolE}{b}{n}{
	<-6>  MnSymbolE-Bold5
	<6-7>  MnSymbolE-Bold6
	<7-8>  MnSymbolE-Bold7
	<8-9>  MnSymbolE-Bold8
	<9-10> MnSymbolE-Bold9
	<10-12> MnSymbolE-Bold10
	<12->   MnSymbolE-Bold12
}{}

\let\lsem\@undefined
\let\rsem\@undefined
\DeclareMathDelimiter{\lsem}{\mathopen}%
{MnLargeSymbols}{'102}{MnLargeSymbols}{'102}
\DeclareMathDelimiter{\rsem}{\mathclose}%
{MnLargeSymbols}{'107}{MnLargeSymbols}{'107}                     

\makeatother

\DeclareMathOperator{\Ind}{Ind}

\usepackage{tikz}
\newcommand{\tikznode}[2]{%
	\ifmmode%
	\tikz[remember picture,baseline=(#1.base),inner sep=0pt] \node (#1) {$#2$};%
	\else
	\tikz[remember picture,baseline=(#1.base),inner sep=0pt] \node (#1) {#2};%
	\fi}
\usepackage{multirow}

\newcommand{\A}{\mathbf{A}}

\newcommand{\C}{\mathbf{C}}

\newcommand{\F}{\mathbf{F}}

\newcommand{\bh}{\mathbf{h}}

\newcommand{\bj}{\mathbf{j}}
\newcommand{\bm}{\mathbf{m}}
\newcommand{\N}{\mathbf{N}}
\newcommand{\Q}{\mathbf{Q}}
\newcommand{\Qp}{\mathbf{Q}_p}
\newcommand{\Qpbar}{\overline{\Q}_p}

\newcommand{\R}{\mathbf{R}}

\newcommand{\Z}{\mathbf{Z}}
\newcommand{\Zp}{\mathbf{Z}_p}
\newcommand{\zp}{\Zp}

\newcommand{\roi}{\mathcal{O}}

\newcommand{\cA}{\mathcal{A}}
\newcommand{\cU}{\mathcal{U}}

\newcommand{\cD}{\mathcal{D}}

\newcommand{\cF}{\mathcal{F}}
\newcommand{\cH}{\mathcal{H}}

\newcommand{\cM}{\mathcal{M}}

\newcommand{\cP}{\mathcal{P}}
\newcommand{\cQ}{\mathcal{Q}}

\newcommand{\cS}{\mathcal{S}}
\newcommand{\cV}{\mathcal{V}}

\newcommand{\cX}{\mathcal{X}}
\newcommand{\cO}{\mathcal{O}}

\newcommand{\sD}{\mathscr{D}}
\newcommand{\sE}{\mathscr{E}}

\newcommand{\sV}{\mathscr{V}}

\newcommand{\sM}{\mathscr{M}}
\newcommand{\sN}{\mathscr{N}}

\newcommand{\sU}{\mathscr{U}}

\newcommand{\fd}{\mathfrak{d}}
\newcommand{\ff}{\mathfrak{f}}
\newcommand{\fg}{\mathfrak{g}}

\newcommand{\fh}{\mathfrak{h}}

\newcommand{\m}{\mathfrak{m}}

\newcommand{\fp}{\mathfrak{p}}
\newcommand{\pri}{\mathfrak{p}}

\newcommand{\fX}{\mathfrak{X}}

\DeclareMathAlphabet{\mathpgoth}{OT1}{tx-frak}{m}{n}

\newcommand{\cl}{\mathrm{Cl}}

\newcommand{\h}{\mathrm{H}}
\DeclareMathOperator{\Hom}{Hom}
\DeclareMathOperator{\GL}{GL}
\DeclareMathOperator{\Gal}{Gal}
\newcommand{\Galp}{\Gal_p}

\newcommand{\coinv}{\mathrm{coinv}}

\newcommand{\htc}{\mathrm{H}^t_{\mathrm{c}}}
\newcommand{\hc}[1]{\mathrm{H}^{#1}_{\mathrm{c}}}

\renewcommand{\leq}{\leqslant}
\renewcommand{\geq}{\geqslant}

\newcommand*{\longhookrightarrow}{\ensuremath{\lhook\joinrel\relbar\joinrel\rightarrow}}
\newcommand*{\longtwoheadrightarrow}{\ensuremath{\relbar\joinrel\twoheadrightarrow}}
\newcommand\isorightarrow{\xrightarrow{
		\,\smash{\raisebox{-0.65ex}{\ensuremath{\scriptstyle\sim}}}\,}}

\newcommand\labelisorightarrow[1]{\xrightarrow[
	\,\smash{\raisebox{0.65ex}{\ensuremath{\scriptstyle\sim}}}\,]{#1}}

\newcommand*{\defeq}{\mathrel{\vcenter{\baselineskip0.5ex \lineskiplimit0pt
			\hbox{\scriptsize.}\hbox{\scriptsize.}}}%
	=}
\newcommand*{\defeqrev}{=\mathrel{\vcenter{\baselineskip0.5ex \lineskiplimit0pt
			\hbox{\scriptsize.}\hbox{\scriptsize.}}}%
}

\newcommand{\matrd}[4]{\begin{pmatrix}#1 & #2\\#3 & #4\end{pmatrix}}
\newcommand{\smallmatrd}[4]{\left(\begin{smallmatrix}#1 & #2\\#3 & #4\end{smallmatrix}\right)}

\newcommand{\newmod}[1]{\hspace{2pt}(\mathrm{mod}\hspace{2pt}#1)}

\newcommand{\beqn}{\begin{eqnarray*}}
	\newcommand{\eeqn}{\end{eqnarray*}}
\newcommand{\beqa}{\begin{eqnarray}}
	\newcommand{\eeqa}{\end{eqnarray}}

\newcommand{\sw}{{\sf w}}

\let\OLDthebibliography\thebibliography
\renewcommand\thebibliography[1]{
	\OLDthebibliography{#1}
	\setlength{\parskip}{0pt}
	\setlength{\itemsep}{0pt plus 0.3ex}
}

\newcounter{lett}

\newtheorem{theorem-intro}[lett]{Theorem}
\newtheorem{conjecture-intro}[lett]{Conjecture}
\newtheorem{question-intro}[lett]{Question}
\newtheorem{proposition-intro}[lett]{Proposition}
\newcounter{lettprime}

\newtheorem{theorem-intro-2}[lettprime]{Theorem}

\newtheorem{thm}{Theorem}[section]

\newtheorem{theorem}[thm]{Theorem}
\newtheorem*{theorem*}{Theorem}
\newtheorem{lemma}[thm]{Lemma}
\newtheorem*{lemma*}{Lemma}
\newtheorem{proposition}[thm]{Proposition}
\newtheorem*{proposition*}{Proposition}
\newtheorem{corollary}[thm]{Corollary}
\newtheorem*{corollary*}{Corollary}
\newtheorem{definition-proposition}[thm]{Definition-Proposition}

\newtheorem{conjecture*}{Conjecture}

\theoremstyle{definition}
\newtheorem{definition-intro}[lettdef]{Definition}

\newtheorem{definition}[thm]{Definition}
\newtheorem{test-data}[thm]{Test Data}
\newtheorem*{definition*}{Definition}
\newtheorem*{remark*}{Remark}
\newtheorem*{remarks*}{Remarks}
\newtheorem{remark}[thm]{Remark}

\newtheorem*{notation*}{Notation}

\newtheorem*{fact*}{Fact}

\numberwithin{equation}{section}

\usepackage{titlesec}
\titleformat{\subsubsection}[runin]
{\normalfont\itshape}
{\thesubsubsection.}
{0.5em}
{}
[.\hspace*{6pt}\nopagebreak]

\titleformat{\subsection}[runin]
{\normalfont\bfseries}
{\thesubsection.}
{0.5em}
{}
[.\hspace*{6pt}\nopagebreak]

\titleformat{\section}
{\large\bfseries}
{\thesection.}
{0.5em}
{}
[\nopagebreak]

\titlespacing*{\subsubsection}
{0pt}{2ex plus .5ex minus .2ex}{.5ex plus .2ex}

\usepackage{comment}

\newcommand{\sar}[2]{\ar@{}[#1]|-*[@]{#2}}

\usepackage{tikz}
\usetikzlibrary{positioning}
\usepackage{mathtools}
\usepackage{environ}
\NewEnviron{myquote}{\vspace{1ex}\par
	\hfill\llap{($\dagger$)}\hfill\parbox{\textwidth-2cm}%
	{\emph{\BODY}}%
	\vspace{1ex}\par}

\NewEnviron{myquote2}{\vspace{0ex}\par
	\hspace{5pt}\parbox{\textwidth-2cm}%
	{\centering\BODY} \hfill\llap{($*$)}\hfill%
	\vspace{1ex}\par}

\usepackage{tocloft}

\cftpagenumbersoff{part}

\titleformat{\part}
{\large\scshape\centering}
{Part \thepart.}
{0.5em}
{}
[]

\begin{document}

	\maketitle

	\renewcommand{\thefootnote}{\fnsymbol{footnote}} 
	\footnotetext{\texttt{chris.williams1@nottingham.ac.uk}.  \today. \emph{2020 MSC:} Primary 11F33,  11F67; Secondary 11R23.  
	}     
	\renewcommand{\thefootnote}{\arabic{footnote}}

	\begin{abstract}
	Let $F$ be a number field, and $\pi$ a regular algebraic cuspidal automorphic representation of $\mathrm{GL}_N(\mathbb{A}_F)$ of symplectic type. When $\pi$ is spherical at all primes $\pri|p$, we construct a $p$-adic $L$-function attached to any regular non-critical spin $p$-refinement $\tilde\pi$ of $\pi$ to $Q$-parahoric level, where $Q$ is the $(n,n)$-parabolic. More precisely, we construct a distribution $L_p(\tilde\pi)$ on the Galois group $\mathrm{Gal}_p$ of the maximal abelian extension of $F$ unramified outside $p\infty$, and show that it interpolates all the standard critical $L$-values of $\pi$ at $p$ (including, for example, cyclotomic and anticyclotomic variation when $F$ is imaginary quadratic). We show that $L_p(\tilde\pi)$ satisfies a natural growth condition; in particular, when $\tilde\pi$ is ordinary, $L_p(\tilde\pi)$ is a (bounded) measure on $\mathrm{Gal}_p$.  	As a corollary, when $\pi$ is unitary, has very regular weight, and is $Q$-ordinary at all $\mathfrak{p}|p$, we deduce  non-vanishing $L(\pi\times(\chi\circ N_{F/\mathbb{Q}}),1/2) \neq 0$ of the twisted central value for all but finitely many Dirichlet characters $\chi$ of $p$-power conductor.
\end{abstract}

	\setcounter{tocdepth}{1}
	\footnotesize
	\tableofcontents
	\normalsize

	\section{Introduction}

The special values of $L$-functions are highly important in modern number theory, and are the subject of a vast network of conjectures, including: 
\begin{enumerate}[(I)]\s
	\item The \emph{Beilinson and Bloch--Kato conjectures}, which describe arithmetic data in terms of the complex analytic properties of special values. As a special case, non-vanishing of a special value should force finiteness of an associated Selmer group.
	\item  \emph{Deligne's conjecture} \cite{Del79}, which predicts that the special values, ostensibly arbitrary transcendental numbers, are algebraic after scaling by controlled complex `periods'. 
	\item The \emph{Coates--Perrin-Riou/Panchishkin conjectures} \cite{coatesperrinriou89, coates89,Pan94}, that predict these special $L$-values satisfy deep $p$-adic congruences, generalising Kummer's congruences for the Riemann zeta function. 
\end{enumerate} 
The congruences predicted by (III) are encoded in the existence of a \emph{$p$-adic $L$-function}. Beautiful objects in their own right,  $p$-adic $L$-functions are fundamental in Iwasawa theory, and have led to substantial results towards (I), including special cases of the Birch--Swinnerton-Dyer conjecture.

A very general setting to consider these questions is that of regular algebraic cuspidal automorphic representations (\emph{RACARs}) $\pi$ of $\GL_N$ over a number field $F$. In this paper, we attack (III) by constructing $p$-adic $L$-functions for those $\pi$ which are \emph{symplectic} (\emph{RASCARs}, for RA-symplectic-CARs).

\subsection{Main result}
Let us state our main result more precisely. Let $F$ be an arbitrary number field, $G \defeq \mathrm{Res}_{F/\Q}(\GL_N)$, and $\pi$ a RASCAR of $G(\A)$. This forces $N = 2n$ to be even, and $\pi \cong \pi^\vee \otimes \eta$ to be essentially self-dual, for a Hecke character $\eta$ of $F^\times\backslash\A_F^\times$. Further, $\pi$ is symplectic if and only if it admits a Shalika model, if and only if it is a functorial transfer from $\mathrm{Res}_{F/\Q}\mathrm{GSpin}_{2n+1}$ \cite{FJ93,AS06,AS14}. In this setting, the automorphic realisation of Deligne's conjecture was recently proved -- including the period relations at infinity -- in work of Jiang--Sun--Tian \cite{JST}. 

Let $p$ be a prime, and assume that $\pi$ is spherical at all primes $\pri$ of $F$ above $p$. Then $\pi_{\pri} = \Ind_B^G \UPS_{\pri}$ is an unramified principal series, the normalised induction of some unramified character $\UPS_{\pri} = (\UPS_{\pri,1},...,\UPS_{\pri,2n})$ (where $B$ is the upper-triangular Borel). Let $Q \subset G$ be the standard parabolic subgroup with Levi $H \defeq \mathrm{Res}_{F/\Q}(\GL_n\times\GL_n)$, and let $J_{\pri} \subset \GL_{2n}(F_{\pri})$ be the parahoric subgroup of type $Q$. If $\pri|p$, a regular $Q$-refinement to level $J_{\pri}$ is a choice of simple Hecke eigenvalue $\alpha_{\pri}$ on $\pi_{\pri}^{J_{\pri}}$, defined precisely in Definition \ref{def:refinement}. We assume this choice is \emph{spin} (Definition \ref{def:spin}), in that it interacts well with the Shalika model; as explained in the introduction of \cite{classical-locus}, we expect this is essential for constructions of $p$-adic $L$-functions via Shalika models. Write $\tilde\pi$ for $\pi$ with such a choice of regular spin $Q$-refinement at each $\pri|p$. We assume $\tilde\pi$ to be non-$Q$-critical (Definition \ref{def:non-critical}). This is satisfied if the integrally normalised eigenvalues $\alpha_{\pri}^\circ$, defined in \S\ref{sec:construction}, have non-$Q$-critical slope. In particular if $\tilde\pi$ is ordinary (i.e. each $v_p(\alpha_{\pri}^\circ) = 0$) then it is non-$Q$-critical.

We say a Hecke character $\chi : F^\times \backslash \A_F^\times \to \C^\times$ is \emph{critical} for $\pi$ if $s=1/2$ is a Deligne-critical value of $L(\pi\times\chi,s)$, and then we define $L(\pi,\chi) \defeq L(\pi\times\chi,1/2)$. The critical characters are determined by their infinity type and the weight $\lambda$ of $\pi$. For general $F$, this $L$-function can have many different critical regions; see e.g.\ the pictures in \cite[p.1605]{LLZ15} or \cite[\S4.1]{BDP13}. We focus here on the `standard' critical range\footnote{Constructions of $p$-adic $L$-functions interpolating other ranges would be extremely interesting; see e.g.\ \cite{BDP13}, which gives an example for $\GL_2$ over an imaginary quadratic field, assuming $\pi$ is base-change.}, corresponding to the `balanced weight' condition in \cite{JST} (and, when $F$ is imaginary quadratic, region $\Sigma^{(1)}$ in \cite{LLZ15}). We describe this range in Lemma \ref{lem:critical weight}. For example, when $F$ is imaginary quadratic, this gives a square of standard critical infinity types (see \S\ref{sec:IQF}). Crucially, these standard critical values admit a representation-theoretic interpretation via branching laws for $H \subset G$, described in Lemma \ref{lem:branching law}.

 We write $\mathrm{Crit}_p(\pi)$ for the set of standard critical characters that have $p$-power conductor. Let $\Galp$ be the Galois group of the maximal abelian extension of $F$ unramified outside $p\infty$. Attached to any $\chi \in \mathrm{Crit}_p(\pi)$ is a canonical character $\chi_{[p]}$ on $\Galp$  (see \S\ref{sec:p-adic Hecke}).
 
We construct a $p$-adic $L$-function attached to $\tilde\pi$, in the following sense:

\begin{theorem-intro}\label{thm:intro}
	For any choice of isomorphism $i_p : \C \isorightarrow \overline{\Q}_p$, there exists a finite extension $L/\Qp$ and an $L$-valued locally analytic distribution $L_p^{i_p}(\tilde\pi)$ on $\Gal_p$ such that:
	\begin{enumerate}[(a)] \s
		\item For any $\chi \in \mathrm{Crit}_p(\pi)$ of conductor $\prod_{\pri|p} \pri^{\beta_{\pri}}$, we have
	\begin{align*}
	 L_p^{i_p}(\tilde\pi, \chi) &\defeq i_p^{-1} \left(\int_{\Galp} \chi_{[p]} \cdot dL_p^{i_p}(\tilde\pi)\right)\\
		&= 	A  \cdot \tau(\chi_f)^n \cdot   \prod_{\pri|p} e(\pi_{\pri},\chi_{\pri}) \cdot \frac{L^{(p)}(\pi,\chi)}{\Omega_{\pi,\chi_\infty}},
	\end{align*}
	where  $A$ is a constant defined in \eqref{eq:A}, $\Omega_{\pi,\chi_\infty} \in \C^\times$ is defined in Definition \ref{def:period}, and
	\[
	e(\pi_{\pri},\chi_{\pri})=\left\{\begin{array}{cl} q_{\pri}^{\beta_{\pri}\left(\tfrac{n^2-n}{2}\right)} \alpha_{\pri}^{-\beta_{\pri}} &: \chi_{\pri} \text{ ramified},\\
		\displaystyle \prod_{i=n+1}^{2n}
		\frac{1-\UPS_{\pri,i}^{-1}\chi_{\pri}^{-1}(\varpi_{\pri})q_{\pri}^{-1/2}}{1-\UPS_{\pri,i}\chi_{\pri}(\varpi_{\pri})q_{\pri}^{-1/2}} &: \chi_{\pri} \text{ unramified}.\end{array}\right.
	\]
	\item $L_p^{i_p}(\tilde\pi)$ is admissible of growth $(h_{\pri})_{\pri|p}$, where $h_{\pri} \defeq v_{p}(\alpha_{\pri}^\circ)$. 
	\end{enumerate} 
\end{theorem-intro}

The growth condition is described in Definition \ref{def:admissible}. In particular, if $\tilde\pi$ is ordinary, then $L_p^{i_p}(\tilde\pi)$ is a \emph{bounded} distribution, that is, a $p$-adic measure. Here $L^{(p)}$ is the $L$-function with the Euler factors at $p$ removed, $\tau(\chi_f)$ is a Gauss sum, $q_{\pri} = N_{F/\Q}(\pri)$, and $\varpi_{\pri}$ is a uniformiser. This agrees exactly with the conjectures of Coates--Perrin-Riou/Panchishkin \cite[Conj.\ 6.2]{Pan94}, except possibly the term $\Omega_{\pi,\chi_\infty}$, which we comment on in Remark \ref{rem:period relations}.

If $F$ is imaginary quadratic, and $\tilde\pi$ has non-$Q$-critical slope, then (a) and (b) determine $L_p^{i_p}(\tilde\pi)$ uniquely by \cite{Loe14}. Analogous unicity results for $F$ totally real are \cite[Prop.\ 6.25]{BDW20}.

\medskip

To our knowledge, Theorem \ref{thm:intro} gives the first such construction beyond the special cases of $\GL_2$ (over any $F$) or $F$ totally real; see \S\ref{sec:special cases} for a summary of previous results. The totally real case was handled in \cite{DJR18,BDW20}. The construction we give is inspired by those works, particularly the overconvergent approach of \cite{BDW20}. There are some additional features in the general number field setting, which we briefly summarise.
\begin{enumerate}[(1)]\s
	\item Notably, whilst Leopoldt's conjecture predicts that $L_p(\tilde\pi)$ is 1-variabled when $F$ is totally real, in general our $p$-adic $L$-functions can be many-variabled. For example, when $F = \Q$ we have $\Galp \cong \Zp^\times$, and one only has cyclotomic variation; when $F$ is imaginary quadratic, one has a two-variable $p$-adic $L$-function, with both cyclotomic and anticyclotomic variation. 

\item The method uses automorphic cycles/modular symbols. The totally real (resp.\ $\GL_2$) constructions relied on a certain numerical coincidence; that these cycles have dimension equal to the top degree (resp.\ bottom degree) in which cuspidal cohomology for $G$ contributes. This rendered some relevant cohomology groups 1-dimensional, making certain choices unique up to scalar. In the general number field case, we work in the middle of the cuspidal range, where the analogous groups are never 1-dimensional.
\end{enumerate}

We handle (1) by blending the methods of \cite{BDW20} with earlier work \cite{Wil17,BW_CJM} of the author and Barrera for $\GL_2$. Problem (2) is more serious, and for this we exploit works of Lin--Tian \cite{LT20} and Jiang--Sun--Tian \cite{JST}. They nailed down good elements in these higher-dimensional cohomology groups and proved not only the non-vanishing hypothesis of attached zeta integrals at infinity, but also the expected period relations at infinity, at least in the cyclotomic direction (see Remark \ref{rem:period relations}). We hope that the period relations in \emph{all} directions can be extracted via similar methods, but do not address that here, instead focusing on the $p$-adic interpolation.

We imagine the following special case might be of particular interest. Let $\cF$ be a cohomological Siegel modular form on $\mathrm{GSp}_4/\Q$, and base-change it to an imaginary quadratic field $F$. Under appropriate assumptions, we may apply our construction to the Langlands transfer of $\cF/F$ to $\GL_4/F$ (via \cite{AS06}), giving a two-variable $p$-adic spin $L$-function for $\cF/F$, including anticyclotomic variation. For $\GL_2$, such constructions for base-change modular forms have had important arithmetic consequences, such as proofs of one divisibility of the Iwasawa Main Conjecture \cite{SU14}.

\subsection{Application: non-vanishing of central twists}

When $\tilde\pi$ is ordinary at each $\pri|p$, then as mentioned above, our construction yields a $p$-adic measure. As an immediate application, we get the following generalisation of a result of Dimitrov--Januszewski--Raghuram (who in \cite{DJR18} treated the case $F$ totally real). Let $\lambda = (\lambda_\sigma)_{\sigma\in\Sigma}$ denote the weight of $\pi$, where $\lambda_\sigma = (\lambda_{\sigma,1},...,\lambda_{\sigma,2n}) \in \Z^{2n}$ is dominant and $\Sigma$ is the set of embeddings $\sigma : F \hookrightarrow \C$ (see \S\ref{sec:weights}).

\begin{theorem-intro}\label{thm:intro non-vanishing}
	Let $\pi$ be a unitary RASCAR of $G(\A)$, and suppose
	\begin{equation}\label{eq:regular weight}
		\lambda_{\sigma,n} > \lambda_{\sigma,n+1} \qquad \text{ for all }\sigma \in \Sigma.		
	\end{equation}
 Suppose there exists a rational prime $p$ such that $\pi_{\pri}$ is spherical and $Q$-ordinary for all $\pri|p$. Then for all but finitely many Dirichlet characters $\chi$ of $p$-power conductor, we have non-vanishing of the twisted central $L$-value
	\[
		L(\pi\times(\chi\circ N_{F/\mathbb{Q}}),\tfrac12) \neq 0.
	\]
\end{theorem-intro}

Here \emph{$Q$-ordinarity} is in the sense of \cite[\S1.1]{HidP-ord}. The unitary assumption ensures $s=1/2$ is a Deligne-critical value of $L(\pi,s)$.

 Given Theorem \ref{thm:intro}, the proof of Theorem \ref{thm:intro non-vanishing} is simple. The argument is literally identical to \cite[\S4.4]{DJR18}, so we give only a sketch, and refer the reader there for full details. 

\begin{proof}
	 Since $\pi_{\pri}$ is $Q$-ordinary for all $\pri|p$, by \cite[Lem.\ 4.4]{DJR18} there is a (unique) regular spin ordinary $Q$-refinement $\tilde\pi$ of $\pi$. Let $L_p^{i_p}(\tilde\pi)$ the $p$-adic $L$-function from Theorem \ref{thm:intro}. By \cite{JS76} (see \cite[Lem.\ 7.4]{BDW20}), the weight condition forces existence of some $j \neq 0$ such that the (non-central) twisted $L$-values $L(\pi\times(\chi\circ N_{F/\Q}),j+1/2)$ are non-zero for all Dirichlet characters $\chi$.
	
	 If $\chi$ is a Dirichlet character of $p$-power conductor, then $\chi \circ N_{F/\Q}$ is a cyclotomic character of $\Galp$. The restriction of $L_p^{i_p}(\tilde\pi)$ to the cyclotomic line is a bounded rigid analytic function on the disjoint union of a finite number of open unit discs.  By the interpolation property and non-vanishing at non-central values, the $p$-adic $L$-function is non-zero on each disc that contains characters of the form $\chi \circ N_{F/\Q}$. By Weierstrass preparation, it has finitely many zeros in each such disc, hence on the cyclotomic line; so the theorem follows from the interpolation property.
\end{proof}

If one drops assumption \eqref{eq:regular weight}, the theorem instead becomes: if there exists a $\chi$ as in the theorem such that $L(\pi \times (\chi\circ N_{F/\mathbb{Q}}),1/2) \neq 0$, then there are infinitely many such $\chi$.

In the $\GL_4$ case, after transferring via \cite{AS06}, this gives non-vanishing of many twisted central values of spin $L$-functions of (very regular weight, cohomological, Klingen-ordinary) genus 2 Siegel modular forms over number fields.

\subsection{Relation to the literature}\label{sec:special cases}

This work generalises (and is visibly inspired by) many other constructions. We summarise a few. In the case of $\GL(2)$, i.e.\ $n=1$, every RACAR is a RASCAR, and our main results/methods specialise exactly to those of \cite{BW_CJM}; and the results/methods of that paper in turn followed the earlier works \cite{PS11,Bar15,Wil17} (over $\Q$, totally real, and imaginary quadratic base fields respectively). These methods were later used in \cite{Bel12, BDJ17,BH17,BW18,BW-Iwasawa} to vary $p$-adic $L$-functions in families. 

In the case of general $\GL(2n)$, when $F$ is totally real our main result specialises exactly to \cite[Thm.\ 6.23]{BDW20}. Earlier constructions in the ordinary situation were given in \cite{AG94,Geh18,DJR18}. This construction was then used in \cite{BDW20,BDGJW} to construct and study $p$-adic families, and -- under some technical hypotheses -- to vary $p$-adic $L$-functions over these families.  

The variation of the present construction in families would be very interesting. The methods of these earlier papers -- which worked in either top or bottom cohomological degree -- do not apply. Fundamental difficulties include: controlling the families; constructing classes in the correct cohomological degree; and showing that these classes interpolate the `good' elements from \cite{JST}.

Rohrlich \cite{Roh89} proved a $\GL_2$ analogue of Theorem \ref{thm:intro non-vanishing}. The case $F$ totally real is \cite{DJR18}. For $\GL_4/\Q$, an analogous result for all weights and unitary CARs was recently proved in \cite{RY-GL4}.

Finally, in Theorems \ref{thm:intro} and \ref{thm:intro non-vanishing} it should be possible to weaken the assumption that $\pi_{\pri}$ is spherical at all $\pri|p$ to $Q$-parahoric-spherical using forthcoming work of Dimitrov--Jorza \cite{DJ-parahoric}. 

\subsection*{Acknowledgements} 
This paper would not exist without my previous collaboration with Daniel Barrera and Mladen Dimitrov, and I thank them wholeheartedly for several years' worth of discussions. I also thank Andy Graham and Andrei Jorza for many highly relevant discussions during our follow-up collaboration, David Loeffler and Frederick Thogersen for their comments and corrections on earlier drafts, and the referee for pushing me to make the article more self-contained. This research was supported by EPSRC Postdoctoral Fellowship EP/T001615/2.

\section{Preliminaries}

\subsection{Notation}\label{sec:notation}
Let $F$ be a number field of degree $d = r+2s$, where $F$ has $r$ real embeddings and $s$ pairs of complex embeddings. Let $\Sigma \defeq \{\sigma : F \hookrightarrow \C\} = \Sigma_{\R} \sqcup \Sigma_{\C}$ be the union of the real and complex embeddings. Each $\sigma \in \Sigma$ has a conjguate $c\sigma$, and $\sigma = c\sigma$ if and only if $\sigma(F) \subset \R$. Fix a rational prime $p$ and an isomorphism $i_p : \C \isorightarrow \overline{\Q}_p$. Let $F^{p\infty}$ be the maximal abelian extension of $F$ unramified outside $p\infty$, and let $\Galp \defeq  \mathrm{Gal}(F^{p\infty}/ F)$ be its Galois group. Let $\psi$ be the  standard non-trivial additive character of $F\backslash\A_F$ (as e.g.\ fixed in \cite[\S4.1]{DJR18}). 

If $v$ is a place of $F$, we write $F_v$ for the completion of $F$ at $v$, $\cO_v$ for its ring of integers, and $\F_v$ for its residue field. We fix a choice of uniformiser $\varpi_v \in \cO_v$. 

Let $G = \mathrm{Res}_{F/\Q}\GL_{2n}$, $B = TN$ be the upper-triangular Borel sugroup in $G$, $T$ the diagonal torus and $N$ the unipotent. Let $H = \mathrm{Res}_{F/\Q}(\GL_n\times\GL_n)$, with $\iota : H \hookrightarrow G$, $(h_1,h_2) \mapsto \smallmatrd{h_1}{}{}{h_2}$. We may abuse notation and write e.g.\ $B(F_{v})$ for the upper-triangular matrices in $\GL_{2n}(F_v)$.

Let $K_\infty=C_\infty Z_\infty \subset G(\R)$, where  $Z_\infty$ is the center and $C_\infty = \mathrm{O}_{2n}(\R)^r \times \mathrm{U}_{2n}(\C)^s$ is the maximal compact subgroup of $G(\R)$. If $A$ is a reductive real Lie group, then $A^\circ$ denotes the connected component of the identity.

Let $Q = HN_Q$ be the standard parabolic with Levi $H$. For a prime $\pri|p$ of $G$, let $J_{\pri} \defeq \{g \in \GL_{2n}(\cO_{\pri}) : g \newmod{\pri} \in Q(\F_{\pri})\} \subset \GL_{2n}(F_{\pri})$ be the parahoric subgroup of type $Q$, and $J_p \defeq \prod_{\pri|p}J_{\pri} \subset G(\Qp)$.

\subsection{RASCARs}
Let $\pi$ be a regular algebraic cuspidal automorphic representation (RACAR) of $G(\A)$. If $\pi$ is essentially-self-dual, i.e.\ there exists some Hecke character $\eta$ such that $\pi^\vee \cong \pi \otimes \eta^{-1}$, then $L(\pi\times\pi^\vee,s) = L(\pi, \mathrm{Sym}^2\otimes\eta^{-1},s)L(\pi,\wedge^2\otimes\eta^{-1}, s)$. This has a simple pole at $s=1$, which occurs either in the symmetric or exterior square $L$-function. Then:

\begin{definition}\label{def:RASCAR}
We say that $\pi$ has \emph{symplectic type}, and call it a \emph{RASCAR} (RA-symplectic-CAR), if any (hence all) of the following equivalent conditions hold:
\begin{itemize}\s
\item[(1)] The exterior square $L$-function $L(\pi,\wedge^2 \otimes \eta^{-1},s)$ has a pole at $s=1$;
\item[(2)] $\pi$ is the functorial transfer of an irreducible generic cuspidal automorphic representation $\Pi$ of $\mathrm{GSpin}_{2n+1}(\A_F)$ with central character $\eta$;
\item[(3)] $\pi$ admits a non-trivial $(\eta,\psi)$-Shalika model.
\end{itemize}
\end{definition}

Recall $\pi$ has a $(\eta,\psi)$-Shalika model if there is an intertwining $\cS_{\psi}^\eta : \pi \hookrightarrow \Ind_{\cS(\A)}^{G(\A)} (\eta\otimes\psi)$, where $\cS = \big\{s(h,X) = \smallmatrd{h}{}{}{h}\smallmatrd{1_n}{X}{}{1_n} : h \in \GL_n, X \in\mathrm{M}_n\big\}$ is the Shalika group and $(\eta\otimes\psi)(s(h,X)) = \eta(\det(h))\cdot \psi(\mathrm{Tr}(X))$. Our conventions are summarised in \cite[\S2.6]{BDW20}. 

In the definition, (1)$\iff$(3) was substantially proved in \cite{JS90}, and (1)$\iff$(2) in \cite{AS06,AS14}. For the equivalence in the exact form above, see \cite[Thm.\ 5.1]{GR13}. 

\subsection{Weights}\label{sec:weights}
Let $X^*(T)$ be the set of algebraic weights for $T$. A general $\lambda \in X^*(T)$ has form $\lambda = (\lambda_\sigma)_{\sigma \in \Sigma}$, where $\lambda_\sigma = (\lambda_{\sigma,1}, ..., \lambda_{\sigma,2n}) \in \Z^{2n}$. Let $X_+^*(T)$ be the set of \emph{dominant} weights, where $\lambda_{\sigma,i} \geq \lambda_{\sigma,i+1}$ for all $\sigma$ and $i$. Attached to any $\lambda \in X_+^*(T)$ is an algebraic representation $V_\lambda$ of $G$ of highest weight $\lambda$, with dual $V_\lambda^\vee$. We can decompose $V_\lambda = \otimes_{\sigma \in \Sigma} V_{\lambda,\sigma}$, $V_{\lambda}^\vee = \otimes_{\sigma \in \Sigma}V_{\lambda,\sigma}^\vee$, where $V_{\lambda,\sigma}$ is the algebraic $\GL_{2n}$-representation of highest weight $\lambda_{\sigma}$, with $G$ acting via $\sigma$.

\begin{definition}\label{def:pure}
	We say $\lambda \in X_+^*(T)$ is \emph{pure} if there exists $\sw \in \Z$ such that 
\[
\lambda_{\sigma,i} + \lambda_{c\sigma,2n+1-i} = \sw \qquad \text{for all }\sigma \text{ and }i.
\] 
We write $X_0^*(T) \subset X_+^*(T)$ for the space of pure dominant weights.
\end{definition}

Let $\pi$ be RASCAR of $G(\A)$. Let $\fg_\infty \defeq \mathrm{Lie}(G(\R))$. As in \cite{Clo90}, there is a unique $\lambda \in X^*_0(T)$, the \emph{weight} of $\pi$, such that
\[
	\h^\bullet(\fg_\infty,K_\infty^\circ; \pi_\infty \otimes V_\lambda^\vee(\C)) \neq 0.
\]
Here  $\lambda$ is pure by \cite[Lem.\ 4.9]{Clo90}. Since $\pi$ is a RASCAR, $\lambda$ is further restricted: by \cite[Thm.\ 2.1]{LT20}, for all $\sigma \in \Sigma$ there exists $\sw_\sigma \in \Z$ such that $\lambda_{\sigma,i} + \lambda_{\sigma,2n+1-i} = \sw_\sigma$ (i.e.\ each $\lambda_\sigma$ is \emph{individually} pure; see also \cite[Prop.\ 2.17]{JST} for the statement in this form). We have $\sw_\sigma = \sw$ for $\sigma \in \Sigma_{\R}$, and $\sw_\sigma + \sw_{c\sigma} = 2\sw$ for $\sigma \in \Sigma_{\C}$. We let $\underline{\sw} = (\sw_\sigma)_{\sigma\in\Sigma} \in \Z^{\Sigma}$.

\subsection{Critical $L$-values}
Let $\pi$ be a RASCAR with standard $L$-function $L(\pi,s)$. For a Hecke character $\chi : F^\times\backslash\A_F^\times \to \C^\times$, let $\pi\times\chi$ denote the RASCAR with $G(\A)$-action twisted by $\chi$. We write $L(\pi,\chi) \defeq L(\pi\times\chi,1/2)$.

\begin{definition}\label{def:critical weight}
Let $\chi$ be an algebraic Hecke character with infinity type $\mathbf{j} = (j_\sigma)_{\sigma \in \Sigma} \in \Z^\Sigma$, and let $\pi$ have weight $\lambda$.	We say $\chi$ is a \emph{standard critical value for $L(\pi,-)$} if
	\begin{equation}\label{eq:j critical}
	-\lambda_{\sigma,n} \leq j_\sigma \leq -\lambda_{\sigma,n+1} \qquad \forall\sigma \in \Sigma.
\end{equation}
In this case, we say $\chi$ is \emph{critical for $\lambda$}. We write $\mathrm{Crit}(\pi)$ for the set of all such $\chi$. 
\end{definition}

\begin{lemma}\label{lem:critical weight}
If $\chi \in \mathrm{Crit}(\pi)$, then $s=1/2$ is a Deligne-critical value of $L(\pi\times\chi,s)$. 
\end{lemma}

\begin{proof}
	The RACAR $\pi\times\chi$ has weight $\lambda + \mathbf{j} = (\lambda_\sigma + (j_\sigma,...,j_\sigma))_\sigma$. By \cite[Prop.\ 2.20]{JST}, $s=1/2$ is a standard critical value of $L(\pi\times\chi,s)$ if and only if $-(\lambda_{\sigma,n} + j_\sigma) \leq 0 \leq -(\lambda_{\sigma,n+1} + j_\sigma)$ for all $\sigma \in \Sigma$. This is equivalent to \eqref{eq:j critical}.
\end{proof}

\begin{remark}
	For $\bj$ in this range, $\lambda +\bj$ is a balanced weight in the sense of \cite{JST}. There are other Deligne-critical values, but for these $\lambda+\bj$ is not balanced.
\end{remark}

Crucially $\mathrm{Crit}(\pi)$ also admits a representation-theoretic description. For $\mathbf{j}_1, \mathbf{j}_2 \in \Z^{\Sigma}$, let $V_{\mathbf{j}_1,\mathbf{j}_2}^H$ be the $H$-representation of highest weight $(\mathbf{j}_1,\mathbf{j}_2)$, that is, $\det_1^{\mathbf{j}_1}\cdot \det_2^{\mathbf{j}_2}$. The following is proved in the same way as \cite[Prop.\ 6.3.1]{GR2} and \cite[Prop.\ 2.20]{JST}, via \cite[Thm.\ 2.1]{Kna01}.

\begin{lemma}\label{lem:branching law}
A Hecke character $\chi$ of infinity type $\mathbf{j} \in \Z^\Sigma$ is in $\mathrm{Crit}(\pi)$  if and only if 
	\[
		\mathrm{Hom}_H(V_\lambda^\vee, V^H_{\mathbf{j},-\usw-\mathbf{j}}) \neq 0.
	\]
In this case, $\mathrm{Hom}_H(V_\lambda^\vee, V^H_{\mathbf{j},-\usw-\mathbf{j}})$ is a line, generated by some $\kappa_{\bj}$.
\end{lemma}

We will interpolate the standard critical $L$-values `at $p$'. As such, let 
\begin{equation}\label{eq:crit p}
	\mathrm{Crit}_p(\pi)\defeq  \{\chi \in \mathrm{Crit}(\pi) : \mathrm{cond}(\chi) = p^{\beta_0}\},
\end{equation}
where $p^{\beta_0} \defeq \prod_{\pri|p}\pri^{\beta_{0,\pri}}$, for $\beta_0 = (\beta_{0,\pri}) \in \Z_{\geq 0}^{\{\pri|p\}}$ a multiexponent.

\subsection{Hecke characters on ray class and Galois groups}\label{sec:p-adic Hecke}
Let $\chi$ be a Hecke character with infinity type $\bj \in \Z[\Sigma]$. We have an algebraic homomorphism $w^{\bj}: F^\times \longrightarrow \overline{\Q}^\times$ given by $w^{\bj}(x) = \prod_{\sigma \in \Sigma}\sigma(x)^{j_\sigma}.$ This then induces maps 
\begin{equation}\label{eq:w^j}
	w^{\bj}_\infty:(F\otimes_{\Q}\R)^\times \rightarrow \C^\times \xrightarrow{i_p}\overline{\Q}_p^\times, \qquad w^{\bj}_p: (F\otimes_{\Q}\Qp)^\times \rightarrow \overline{\Q}_p^\times.
\end{equation}
\begin{definition}
	We define $\chi_{[p]}$ to be the function
	\begin{align*}
	\chi_{[p]}: \A_F^\times \longrightarrow \overline{\Q}_p^\times, \qquad 
		x &\longmapsto \chi(x)\cdot \big[w_\infty^{\bj}(x_\infty)\big]^{-1} \cdot w_p^{\bj}(x_p) = \chi_\infty(\mathrm{sgn}(x_\infty))\cdot \chi_f(x) \cdot w_p^{\bj}(x_p),
	\end{align*}
where $\mathrm{sgn}(x_\infty) \in \{\pm1\}^{\Sigma_{\R}}$. Via \cite{Wei56} this takes values in some finite extension $L/\Qp$. If $\ff$ is the conductor of $\chi$, then (as in e.g.\ \cite[Prop.\ 2.4]{BW_CJM}) $\chi_{[p]}$ is a locally analytic character on the $p$-adic analytic group 
\[
\cl_F^+(\ff p^\infty) \defeq F^\times \backslash \A_F^\times/U(\ff p^\infty)F_\infty^+,
\]
where $U(\ff p^\infty)$ is the group of elements of $\widehat{\roi}_F^\times$ that are congruent to $1\newmod{\ff p^\infty}$. 
\end{definition}

Recall $\Galp$ is the Galois group for the maximal abelian extension of $F$ unramified outside $p\infty$. The structure of this group is described in detail in \cite[\S6.1]{BDW20}; in particular, the Artin reciprocity map induces an isomorphism $\cl_F^+(p^\infty) \isorightarrow \Galp$, and we have an exact sequence
\begin{equation}\label{eq:cl SES}
	0 \to \overline{\cO}_{F,+}^\times \to (\cO_F\otimes\Zp)^\times \to \Galp \to \cl_F^+ \to 0.
\end{equation}
Via reciprocity, we may consider $\chi_{[p]}$ as a character of $\Galp$.

\subsection{Locally symmetric spaces and local systems}

If $K \subset G(\A_f)$ is open compact, the locally symmetric space of level $K$ is
\[
S_K \defeq G(\Q)\backslash G(\A)/KK_\infty^\circ.
\]

\label{sec:local systems}
If $M$ is a left $G(\Q)$-module, then we have an attached `archimedean' local system $\cM$ on $S_K$ given  by the locally constant sections of 
\begin{equation}\label{eq:local system}
G(\Q)\backslash[G(\A)\times M]/KK_\infty^\circ \longrightarrow S_K,
\end{equation}
where $\gamma(g,m)kz = (\gamma gkz, \gamma\cdot m)$ for $\gamma \in G(\Q), g\in G(\A), k\in K$, $z \in K_\infty^\circ$ and $m \in M$.
Similarly, if $M$ is a left $K$-module, then we have an attached `$p$-adic' local system on $S_K$, given by the local constant sections of \eqref{eq:local system}, but with action $\gamma(g,m)kz = (\gamma gkz, k^{-1}\cdot m)$. 

If $M$ is a left $G(\Qp)$-module, then $G(\Q)$ acts via $G(\Q)\subset G(\Qp)$, and $K$ acts via projection to $G(\Qp)$. We get two attached local systems $\cM$ and $\sM$, and these are isomorphic via the map $(g,m) \mapsto (g_p^{-1}\cdot m)$ of local systems. For more detail see \cite[\S2.3]{BDW20}.


\subsection{Algebraic and analytic coefficient modules}

\subsubsection{Algebraic coefficients}\label{sec:algebraic coeffs}
If $\lambda \in X^\bullet(T)$ is a dominant algebraic weight, recall $V_\lambda$ is the algebraic $G$-representation of highest weight $\lambda$. If $\Qp \subset L \subset \overline{\Q}_p$, we describe $V_\lambda(L)$ via algebraic induction $\mathrm{Ind}_{\overline{B}}^{G}\lambda$; namely,
\[
V_\lambda(L) \defeq \{f :G(\Qp) \to L : f\text{ is algebraic}, f(\overline{b}g) = \lambda(\overline{b})f(g)\ \  \forall \overline{b} \in \overline{B}(\Qp)\},
\]	
where $\overline{B}$ is the opposite Borel. This space carries a natural left action of $G(\Qp)$ given by 
\begin{equation}\label{eq:dot action}
	(g\cdot f)(h) \defeq f(hg), \qquad g,h \in G(\Qp),
\end{equation} 
inducing a dual left action on $V_\lambda^\vee(L)$ by $g \cdot \mu(f) \defeq \mu(g^{-1}\cdot f)$. By \S\ref{sec:local systems}, we have attached local systems $\cV_\lambda^\vee(L)$ and $\sV_\lambda^\vee(L)$, and a natural isomorphism $\hc{\bullet}(S_K,\cV_\lambda^\vee(L)) \isorightarrow \hc{\bullet}(S_K,\sV_\lambda^\vee(L))$.

For our $p$-adic interpolation, the subgroup $H \subset G$ plays a crucial role; its importance is that $G/H$ is a spherical variety. It is thus convenient to consider $V_\lambda$ as a double algebraic induction, namely
\begin{equation}\label{eq:double induction}
	V_\lambda \cong \Ind_{\overline{Q}}^{G} \Ind_{\overline{B}\cap H}^{H} \lambda
\end{equation}
(see \cite[Lem.\ 3.6]{BDW20}). Here $\overline{Q}$ is the opposite parabolic to $Q$ (that is, the $(n,n)$-block lower-triangular subgroup).  Write $V_\lambda^H = \Ind_{\overline{B}\cap H}^{H} \lambda$ for the algebraic $H$-representation of highest weight $\lambda$, and denote the standard action of $h \in H$ on $v \in V_{\lambda}^H$ by $\langle h\rangle_\lambda \cdot v$. Then precisely, \eqref{eq:double induction} means we can identify $V_\lambda(L)$ with the space of algebraic $f : G(\Qp) \to V_\lambda^H(L)$ satisfying
\begin{equation}\label{eq:transform V_lambda}
 f(\overline{n}hg) = \langle h \rangle_\lambda \cdot f(g) \qquad \text{for all } \overline{n} \in \overline{N}_Q(\Qp), h\in H(\Qp), g\in G(\Qp).
\end{equation}

\subsubsection{Parahoric analytic coefficients}\label{sec:distributions}
Recall $J_p \subset G(\Qp)$ is the $Q$-parahoric subgroup. The theory of distributions on $J_p$ was developed in \cite{BW20}, and -- in our setting -- described in detail in \cite[\S3]{BDW20}. In \eqref{eq:double induction}, we replace the second algebraic induction with locally analytic induction, i.e.\ we let 
\[
	\cA_\lambda(L) = \mathrm{LAInd}_{\overline{Q}(\Zp)\cap J_p}^{J_p} V_\lambda^H(L),
\]
whence $\cA_\lambda(L)$ is the space of \emph{locally analytic} $f : J_p \to V_\lambda^H(L)$ such that 
\begin{equation}\label{eq:parahoric transform}
	f(\overline{n}hg) = \langle h \rangle_\lambda \cdot f(g) \qquad \text{for all } \overline{n} \in \overline{N}_Q(\Qp), h\in H(\Qp), g\in J_p.
\end{equation}
 Let $\cD_\lambda(L) = \Hom_{\mathrm{cts}}(\cA_\lambda(L),L)$ be the continuous dual.

Dualising the natural inclusion $V_\lambda(L) \subset \cA_\lambda(L)$ yields a specialisation map 
\begin{equation}\label{eq:r_lambda}
	r_\lambda : \cD_\lambda(L) \twoheadrightarrow V_{\lambda}^\vee(L).
\end{equation}

\subsubsection{The $*$-action on distributions}\label{sec:*-action}

The space $\cA_\lambda(L)$ carries a left $J_p$-action by
\[
(g * f)(h) = f(hg), \qquad h,g \in J_p.
\]
Recall $\varpi_{\pri}$ is a uniformiser for $F_{\pri}$, and let 
\begin{equation}\label{eq:t_p}
	t = t_{\pri} \defeq \iota(\varpi_{\pri} I_n, I_n) = \smallmatrd{\varpi_{\pri} I_n}{}{}{I_n} \in \GL_{2n}(F_{\pri}).
\end{equation}
If $f \in \cA_\lambda$, then the function $(t_{\pri}^{-1} * f) : N_Q(\Zp) \to V_\lambda^H(L)$ defined by $(t_{\pri}^{-1} * f)(n) \defeq f(t_{\pri}nt_{\pri}^{-1})$ extends uniquely to an element of $\cA_\lambda$ (recalling $N_Q(\Zp) \subset Q(\Zp)$ is the unipotent radical). We get an induced dual action on $\cD_\lambda$ given by $(t_{\pri} * \mu)(f) = \mu(t_{\pri}^{-1} * f)$, as in \cite[\S3.4]{BDW20}. 

Let $\Delta_p \subset G(\Qp)$ (resp.\ $\Delta_p^{-1}$) be the semigroup generated by $J_p$ and $t_{\pri}$ (resp.\ $t_{\pri}^{-1}$) for $\pri|p$. The actions above extend to actions of $\Delta_p^{-1}$ on $\cA_\lambda$ and $\Delta_p$ on $\cD_\lambda$.

\subsubsection{Intertwining of actions}\label{sec:actions}

The $*$-action of $\Delta_p^{-1}$ on $\cA_\lambda(L)$ preserves the subspace $V_\lambda(L)$. In particular, the $\Delta_p$-action on $\cD_\lambda(L)$ descends to the quotient $V_\lambda^\vee(L)$. Moreover, this $*$-action preserves  natural integral subspaces in all of these spaces (see \cite[Rem.\ 3.13]{BDW20}). 

As $\Delta_p \subset G(\Qp)$, \eqref{eq:dot action} induces a $\cdot$-action of $\Delta_p$ on $V_\lambda^\vee(L)$. The $\cdot$- and $*$-actions of $J_p$ visibly agree; but the actions of $t_{\pri}^{-1}$ differ. From the definitions, if $\mu \in V_\lambda^\vee(L)$, then
\begin{equation}\label{eq:dot vs *}
	t_{\pri} * \mu = \lambda(t_{\pri}) \times (t_{\pri}\cdot \mu).
\end{equation}

Since the $\cdot$- and $*$-actions of $J_p$ agree, if $K = K^pJ_p \subset G(\A_f)$ is a parahoric-at-$p$ level group, then both actions yield the same local system $\sV_\lambda^\vee(L)$ on $S_K$. However, since the actions of $t_{\pri}$ are different, we get two different Hecke operators on the cohomology $\hc{\bullet}(S_K,\sV_\lambda^\vee(L))$: the `automorphic' operator $U_{\pri}^\cdot$, which is canonical but may have non-integral eigenvalues; and the `$p$-adic' operator $U_{\pri}^*$, which is integrally normalised but depends on the choice of $\varpi_{\pri}$. 


\subsection{Automorphic cohomology classes}
By \cite[p.120]{Clo90}, if $\pi$ has weight $\lambda$ then the $(\fg_\infty,K_\infty^\circ)$-cohomology $\h^i(\fg_\infty,K_\infty^\circ; \pi_\infty \otimes V_\lambda^\vee(\C))$ is non-zero  
	exactly when 
\begin{equation}\label{eq:cuspidal range}
	rn^2 + s(2n^2-n)\leq i \leq r(n^2+n-1) + s(2n^2+n-1).
\end{equation}

Via cuspidal cohomology (see \cite{Clo90}), for any open compact $K \subset G(\A_f)$ there is an injective and Hecke-equivariant map
\[
\h^i(\fg_\infty,K_\infty^\circ; \pi_\infty \otimes V_\lambda^\vee(\C)) \otimes \pi_f^K \longhookrightarrow \hc{i}(S_K,\cV_\lambda^\vee(\C)).
\]
Composing with $i_p : \C \isorightarrow \overline{\Q}_p$, and the isomorphism from \S\ref{sec:local systems}, we get an injective map
\begin{equation}\label{eq:automorphic cohomology}
\h^i(\fg_\infty,K_\infty^\circ; \pi_\infty \otimes V_\lambda^\vee(\C)) \otimes \pi_f^K \longhookrightarrow \hc{i}(S_K,\cV_\lambda^\vee(\overline{\Q}_p)) \isorightarrow \hc{i}(S_K,\sV_\lambda^\vee(\overline{\Q}_p)).
\end{equation}


\section{Local test data}\label{sec:local}
To get useful cohomology classes from the map \eqref{eq:automorphic cohomology}, we must make good choices of local input data, that is, elements $\varphi_v \in \pi_v$ for all finite $v$, and a class $[\omega]$ in the $(\fg_\infty,K_\infty^\circ)$-cohomology. We do so here; in \S\ref{sec:FJ finite} for $v\nmid p\infty$, in \S\ref{sec:local at p} for $v|p$, and \S\ref{sec:local zeta infinity} at infinity. 

Since our motivation is global, let $\pi$ be a RASCAR with an $(\eta,\psi)$-Shalika model. Let $\lambda$ be the weight of $\pi$, with purity weight $\usw \in \Z^\Sigma$. Let $\chi \in \mathrm{Crit}_p(\pi)$ be a Hecke character of infinity type $\bj$ (and $p$-power conductor). Let $V_{\chi,s}^H$ be the character $\chi|\cdot|^s \otimes (\chi\eta|\cdot|^s)^{-1}$ of $H(\A)$. We write $V_{\chi_v,s}^H$, $V_{\chi_f,s}^H$ and $V_{\chi_\infty,s}^H$ for the restrictions to $H(F_v)$, $H(\A_{f})$ and $H(\R)$.

 For each place $v$ of $F$, fix a local intertwining $\cS_{\psi_v}^{\eta_v}$ of $\pi_v$ into its Shalika model. This induces finite and infinite intertwinings $\cS_{\psi_f}^{\eta_f}$ and $\cS_{\psi_\infty}^{\eta_\infty}$.

\subsection{Friedberg--Jacquet integrals at finite places}\label{sec:FJ finite}
First we recap \cite[\S3]{FJ93} (see also \cite[Prop.\ 2.10]{JST}). Let $v$ be a finite place, and fix a Haar measure $dg_v$ on $\GL_n(F_v)$ such that $\GL_n(\cO_v)$ has volume 1. The \emph{local Friedberg--Jacquet integral} is
\[
Z_v(\varphi_v,\chi_v,s) \defeq \int_{\GL_n(F_v)}\cS_{\psi_v}^{\eta_v}(\varphi_v)\left[\matrd{g}{}{}{1}\right]\chi_v|\cdot|^{s-\tfrac{1}{2}}\Big(\det(g)\Big) dg.
\]
This converges absolutely for $\mathrm{Re}(s) \gg 0$. Its normalisation
\[
Z_v^\circ(\varphi_v,\chi_v,s) \defeq \frac{1}{L(\pi_v\times\chi_v,s)}Z(\varphi_v,\chi_v,s) \in \mathrm{Hom}_{H(F_v)} \Big(\pi_v \otimes V_{\chi_v,s-1/2}^H, \ \C\Big)
\]
admits analytic continuation to $\C$.

The following are proved in Propositions 3.1 and 3.2 of \cite{FJ93} respectively (cf.\ \cite[Prop.\ 3.3]{DJR18}). Let $\delta_v$ be the valuation of the different of $F_v/\Q_\ell$, where $v|\ell$, and let $q_v$ be the size of $\F_v$.
\begin{proposition}[Friedberg--Jacquet]\label{prop:FJ test vector}
	\begin{itemize}\s
		\item[(i)]  There exists a test vector $\varphi_v^{\mathrm{FJ}} \in \pi_v$ such that 
		\begin{equation}\label{eq:FJ test vector}
			Z_v^\circ(\varphi_v^{\mathrm{FJ}},\chi_v,s) = (q_v^{s-1/2}\chi_v(\varpi_v)^{-1})^{n\delta_v}
		\end{equation}
		 for all $s$ and all unramified $\chi_v$.
		\item[(ii)] If $\pi_v$ is spherical, the spherical vector is such a test vector.
	\end{itemize}
\end{proposition}
Note we do \emph{not} get 1 on the right-hand side of \eqref{eq:FJ test vector}, as our $\psi_v$ need not have conductor $\cO_v$.

If $\varphi_f \in \pi_f$ can be written as $\varphi_f = \otimes_v \varphi_v$, then define $Z_f^\circ(\varphi_f,\chi_f,s) \defeq \prod_{v\nmid \infty} Z_v^\circ(\varphi_v,\chi_v,s)$ (and similarly $Z_f = \prod Z_v$). Note that 
\[
Z_f^\circ(-,\chi_f,s) \in \mathrm{Hom}_{H(\A_{F,f})}(\pi_f\otimes V_{\chi_f,s-1/2}^H, \C).
\] 
Let $\varphi_{f}^{\mathrm{FJ},(p)} = \otimes_{v\nmid p\infty}\varphi_v^{\mathrm{FJ}} \in \pi_f^{(p)}$. If $\chi \in \mathrm{Crit}_p(\pi)$ and $v \nmid p\infty$, then $\chi_v$ is unramified, hence:
\begin{proposition}\label{prop:FJ}
	For $\pri|p$, let $\varphi_{\pri} \in \pi_{\pri}$ be arbitrary, and let $\varphi_f = \varphi_{f}^{\mathrm{FJ},(p)} \otimes (\otimes_{\pri|p}\varphi_{\pri}) \in \pi_f$. Then for any $\chi \in \mathrm{Crit}_p(\pi)$, and for $\mathrm{Re}(s) \gg 0$, we have
	\begin{equation}\label{eq:friedberg-jacquet}\textstyle
		Z_f(\varphi_f, \chi_f,\tfrac12) =  \left(\prod_{v\nmid p\infty}\chi(\varpi_v^{\delta_v})^{-n}\right) \cdot \left(\prod_{\pri|p}Z_{\pri}\big(\varphi_{\pri},\chi_{\pri},\tfrac12\big)\right) \cdot L^{(p)}\big(\pi_f\otimes\chi_f,\tfrac12\big),
	\end{equation}
	where $L^{(p)}$ is the $L$-function without the Euler factors at $\pri|p$.
\end{proposition}

\subsection{$Q$-refinements and local choices at $\pri|p$}\label{sec:local at p}
For $p$-adic interpolation, it is essential to treat the primes $\pri|p$ separately. In this case, good local vectors have been pinned down in \cite[\S3]{DJR18}, \cite[\S2.7]{BDW20}, \cite[\S9]{BDGJW}, which we summarise here.  \emph{For ease of notation, for \S\ref{sec:local at p} only, we will largely drop subscripts $\pri$, i.e.\ write $\pi = \pi_{\pri}$, $\varpi = \varpi_{\pri}$, $J = J_{\pri}$, etc.).} 

Assume that $\pi$ is spherical. We write $\pi = \Ind_B^G\UPS$, with $\UPS = (\UPS_1,...,\UPS_{2n})$ as in \cite[\S6.1, Notation 9.1]{BDGJW}. In particular we take $\UPS$ such that $\theta_i\theta_{n+i} = \eta$ for $i = 1,\dots,n$ (which is possible by \cite[Prop.\ 1.3]{AG94}, since $\pi$ has an $(\eta,\psi)$-Shalika model).

Recall $t = t_{\pri}$ from \eqref{eq:t_p}. On $\pi^{J}$, we have the Hecke operator  $U_{\fp} := [J t  J]$. 

\begin{definition}\label{def:refinement}
	A \emph{$Q$-refinement} $\tilde\pi = (\pi,\alpha)$ of $\pi$ is a choice of $U_{\pri}$-eigenvalue $\alpha$ on $\pi^J$. We say $\tilde\pi$ is \emph{regular} if $\alpha$ is a simple eigenvalue.
\end{definition}

Recalling that $\pi = \mathrm{Ind}_B^G\theta$, by \cite[Lem.\ 4.8.4]{Che04} (cf.\ \cite[Prop.\ 2.5]{BDGJW}) any $Q$-refinement is of the form
\begin{equation}\label{eq:Itildepi}
\alpha = q^{n^2/2}\prod_{i\in I(\tilde\pi)}\theta_i(\varpi),
\end{equation}
where $q = q_{\pri} \defeq \#\F_{\pri}$ and $I(\tilde\pi) \subset \{1,\dots,2n\}$ is a subset of size $n$. In particular, there are at most ${2n \choose n}$ possible $Q$-refinements. Only a subset of these `interact well with the Shalika model', and in practice we are interested only in this special class of refinements.

\begin{definition}\label{def:spin}
	Let $\tilde\pi = (\pi,\alpha)$ be a regular $Q$-refinement, with corresponding subset $I(\tilde\pi) \subset \{1,\dots,2n\}$ as in \eqref{eq:Itildepi}. We say $\tilde\pi$ is \emph{spin} if $I(\tilde\pi)$ contains exactly one of $1$ or $n+1$, and exactly one of $2$ or $n+2$, $\dots$, and exactly one of $n$ or $2n$. (In particular, if $I(\tilde\pi)$ contains some $i \in \{1,\dots,n\}$, then $I(\tilde\pi)$ does \emph{not} contain $n+i$).
\end{definition}

Note that, in the case where the ${2n \choose n}$ $Q$-refinements are all different, exactly $2^n$ of them are spin. We henceforth fix a regular spin $Q$-refinement $\tilde\pi = (\pi,\alpha)$. For this $\tilde\pi$, up to reordering the $\UPS_i$, we may (and henceforth do) assume $I(\tilde\pi) = \{n+1,n+2,\dots,2n\}$, whence by \eqref{eq:Itildepi} we have
\begin{equation}\label{eq:alpha}
	\alpha = q^{n^2/2}\UPS_{n+1}(\varpi)\cdots\UPS_{2n}(\varpi).
\end{equation}

\begin{remark}
The above criterion for spin refinements was first given in \cite[\S3.3]{DJR18}, where a regular spin $Q$-refinement was called `$Q$-regular'. 

A more conceptual reinterpretation of Definition \ref{def:spin} was given in \cite[Part II]{BDGJW} and \cite{classical-locus}. We summarise briefly. Recall $\pi$ is the functorial transfer of a representation $\Pi$ of $\mathrm{GSpin}_{2n+1}(F_{\pri})$. Via root data, there is a natural parabolic subgroup $\cQ \subset \mathrm{GSpin}_{2n+1}$ associated to $Q\subset G$ (described in \cite[\S2.1]{classical-locus}). Then $\tilde\pi$ is spin if and only if $\alpha$ is an eigenvalue of a Hecke operator $\cU_{\pri}$ on the $\cQ$-parahoric invariant vectors in $\Pi$; see \cite[\S3]{classical-locus} for more details (noting spin here is called $Q$-spin there).

We indicate how spin refinements are expected to `interact well with the Shalika model'. Let $\tilde\pi = (\pi,\alpha)$ be any regular $Q$-refinement, and let $W \in \cS_{\psi}^\eta(\pi^J)[U_{\fp}-\alpha]$ be a generator of the $\alpha$-eigenspace in the Shalika model. Then we expect $W(t^{-\delta}) \neq 0$ if and only if $\tilde\pi$ is spin. The `if' direction was proved in \cite[Lem.\ 3.6]{DJR18}; the converse is predicted by \cite[\S8]{classical-locus} (which also gives partial results towards this). 
\end{remark}
	
The expression $W(t^{-\delta})$ in the above remark is related to (twisted) local Friedberg--Jacquet integrals (see e.g.\ \cite[Prop.\ 3.4]{DJR18}).  Via a separate treatment, these twisted integrals are the subject of Proposition \ref{prop:zeta p} below. They are familiar in the theory of $p$-adic $L$-functions; see, for example, \S\ref{sec:evaluations} of the present paper.

The following `local Birch lemma' is \cite[Prop.\ 9.3]{BDGJW}. Let $u \defeq \smallmatrd{1}{w_n}{0}{1} \in \GL_{2n}(F_{\pri})$, where $w_n$ is the longest Weyl element for $\GL_n$. Recall $\delta = \delta_{\pri}$ is the valuation of the different of $F_{\pri}$. For a smooth character $\chi$ of $F_{\pri}^\times$ of conductor $\fp^{\beta_0}\subset \cO$, let $\tau(\chi)$ be the local Gauss sum 
\[
\tau(\chi) = \tau(\chi,\psi) = q^{\beta_0}(1-q^{-1})\int_{\cO^{\times}} \chi(c\varpi^{-\beta_0-\delta}) \psi(c\varpi^{-\beta_0-\delta}) d^{\times} c,
\]
where $d^\times c$ is the Haar measure on $\cO^\times$ of total measure 1.

\begin{proposition}\label{prop:zeta p}
	Let $\tilde\pi = (\pi,\alpha)$ be a regular spin $Q$-refinement, with $\alpha$ normalised as in \eqref{eq:alpha}. There exists an eigenvector $\varphi^{\alpha} \in \pi^{J}[U_{\pri} - \alpha]$ such that for any smooth character $\chi$ of $F_{\pri}^\times$ of conductor $\pri^{\beta_0} \subset \cO_{\pri}$, we have 
	\[
	Z_{\pri}\Big((u^{-1}t^{\beta}) \cdot \varphi^{\alpha}, \chi, \tfrac12\Big)
	= \frac{q^n}{(q-1)^n} \cdot  \frac{q^{\delta(n^2-n)/2}}{\alpha^\delta} \cdot \chi(\det w_n) \cdot Q(\pi,\chi),
	\]
	where  $\beta = \mathrm{max}(1,\beta_0)$ and
	\[
	Q(\pi,\chi)=\left\{\begin{array}{cl}  q^{\beta\left(\tfrac{-n^2-n}{2}\right)} \cdot \tau(\chi)^n  &: \chi \text{ ramified},\\
		\displaystyle { \chi(\varpi)^{-n\delta} \cdot \alpha \cdot q^{-n^2} \cdot \prod_{i=n+1}^{2n}
			\frac{1-\UPS_{i}^{-1}\chi^{-1}(\varpi)q^{-1/2}}{1-\UPS_{i}\chi(\varpi)q^{-1/2}}} &: \chi \text{ unramified}.\end{array}\right.
	\]
\end{proposition}

\begin{proof}
	This is slightly rearranged from \cite{BDGJW}; the only significant difference in the ramified case is that we exploit \eqref{eq:alpha}. 	In the unramified case, we have pulled a factor of $\UPS_i\chi(\omega)q^{1/2}$ out of the product in $Q(\pi,\chi)$ \emph{op.\ cit.}, and  again used \eqref{eq:alpha}. To make our final formula more `symmetric' in $\beta$ and $\delta$, our eigenvector is $q^{\delta n^2}$ times the eigenvector from \cite{BDGJW}.
\end{proof}

If $\tilde\pi = (\pi,\alpha)$ is not spin, then by contrast we expect this zeta integral vanishes identically on the $\alpha$-eigenspace (see \cite[\S8]{classical-locus}).

\subsection{Classes at infinity}\label{sec:local zeta infinity}
We now choose a class $[\omega]$ at infinity, via the `zeta integral at infinity'. The choice connects closely to the automorphic cycles in \S\ref{sec:critical L-values} (cf.\ \cite[\S4.1]{JST}). 

\subsubsection{The zeta integral at infinity}
Recall $\iota : H \hookrightarrow G$ and $K_\infty$ from \S\ref{sec:notation}, let $L_\infty \defeq H(\R) \cap K_\infty$, and let $\cX_H \defeq H(\R)^\circ/L_\infty^\circ$. Let $\fX_H \defeq \mathrm{Lie}(\cX_H)$, and let $t \defeq \mathrm{dim}_{\R}(\fX_H)$. Letting $\fX_H^{\C} \defeq \fX_H \otimes_{\R} \C$, we have $\wedge^t \fX_H^{\C} = \C$. Given a $\wedge^t\fX_H^{\C}$-valued Haar measure $\mu$ on $\GL_n(F\otimes \R)$, the attached \emph{Friedberg--Jacquet integral at infinity} 
\[
Z_\mu^\circ(-,\chi_\infty,s) \in \mathrm{Hom}_{H(\R)^\circ}\Big( (\wedge^t\fX_H^{\C})^* \otimes \pi_\infty \otimes V_{\chi_\infty,s-1/2}^H, \ \C\Big)
\]
is defined when $\mathrm{Re}(s)\gg0$ by
\begin{align*}
	Z_\mu^\circ(&\omega \otimes \varphi_\infty \otimes 1,\chi_\infty,s) \defeq\\
	& \frac{1}{L(\pi_\infty\otimes\chi_\infty,s)} \int_{\GL_n(F\otimes\R)}
	\cS_{\psi_\infty}^{\eta_\infty}(\varphi_\infty)\left[\matrd{g}{}{}{1}\right]\chi_\infty|\cdot|^{s-\tfrac{1}{2}}\Big(\det(g)\Big) d\langle \mu,\omega\rangle g.
\end{align*}

\subsubsection{Modular symbols at infinity}
The following is \cite[\S4.3]{JST}. As $H(\R)^\circ$-modules, we have 
\[
V_{\chi_\infty}^H \defeq V_{\chi_\infty,0}^H = \chi_\infty \otimes \eta_\infty^{-1}\chi_\infty^{-1} \cong V_{\bj,-\usw-\bj}^H(\C).
\]
In particular, our choice $\kappa_{\bj} : V_\lambda^\vee  \to \VH$ (from Lemma \ref{lem:branching law}) induces an $H(\R)^\circ$-module map $V_\lambda^\vee(\C) \to V_{\chi_\infty}^H$, and hence an $H(\R)^\circ$-module map $V_\lambda^\vee(\C) \otimes (V_{\chi_\infty}^H)^\vee \to \C$. We thus obtain a map
\[
	Z_\mu^\circ(-,\chi_\infty,1/2) \otimes \kappa_{\bj}  : [\pi_\infty \otimes V_{\chi_\infty}^H] \otimes [V_\lambda^\vee(\C)\otimes (V_{\chi_\infty}^H)^\vee] \longrightarrow \wedge^t\fX_H^{\C}.
\]

\begin{definition}
 Let $\fh = \mathrm{Lie}(H(\R))$. The \emph{modular symbol at infinity} is the composition
	\begin{align*}
		\cP_{\chi_\infty} : \h^t(\fg,K_\infty^\circ; \pi_\infty \otimes V_\lambda^\vee(\C)) & \xrightarrow{ \ \iota^* \ } \h^t(\fh,L_\infty^\circ; \pi_\infty \otimes V_\lambda^\vee(\C))\\
		&= \h^t(\fh,L_\infty^\circ; [\pi_\infty \otimes V_{\chi_\infty}^H] \otimes [V_\lambda^\vee(\C)\otimes (V_{\chi_\infty}^H)^\vee])\\
		&\xrightarrow{ \ \  Z_\mu^\circ(-, \chi_\infty,1/2) \otimes \kappa_{\bj}\ \ } \h^t(\fh,L_\infty^\circ;\wedge^t\fX_H^{\C}) = \C.
	\end{align*}
\end{definition}

\begin{theorem}[Lin--Tian, Jiang--Sun--Tian] \label{thm:non-vanishing}
	There exists a class $[\omega] \in \h^t(\fg,K_\infty^\circ; \pi_\infty \otimes V_\lambda^\vee(\C))$ such that $\cP_{\chi_\infty}([\omega]) \neq 0$ for all $\chi \in \mathrm{Crit}(\pi)$.
\end{theorem}

\begin{proof}
	After replacing $\pi$ (globally) with $\pi\times \widetilde{\chi}$ for some (fixed) $\widetilde{\chi} \in \mathrm{Crit}(\pi)$, and modifying $\lambda$ accordingly, we may assume the weight of $\pi$ is balanced in the sense of \cite{JST} (since then $j=0$ is a critical integer). Let $[\omega]$ be the class chosen by Jiang--Sun--Tian after Proposition 4.9 of  \cite{JST}. The required non-vanishing is shown \emph{op.\ cit}.\ for $\chi$ of the form $\chi = \chi_0|\cdot|^j$, with $\chi_0$ finite-order and $j \in \Z$; that is, for $\bj$ parallel. If $F$ has a real embedding, this accounts for all possible critical $\chi$, and we are done. 
	
	If $F$ has no real embedding, then $\chi$ can have non-parallel infinity type. To see the more general non-vanishing we state here, we look at the proof in \cite{JST}. The choice of $[\omega]$ stems from the choice of map $\phi_0$ in Lemma 5.12 \emph{op.\ cit}. Once this choice is fixed, non-vanishing follows from non-vanishing of the composition of three maps in Lemma 5.13 \emph{op.\ cit}. The first two maps are independent of $\chi_\infty$, so are non-vanishing as \emph{op.\ cit}. Non-vanishing of the third follows from non-vanishing of $Z^\circ_\mu(-,\chi_\infty,1/2)$ on the minimal $K$-type in $\pi_\infty$. But this was shown by Lin--Tian in \cite[Thm.\ 1.2]{LT20}. We are in situation (2) of that theorem since $\chi$ is a standard critical character.
\end{proof}

\begin{definition}\label{def:period}
	Define the \emph{period at $\chi_\infty$} (attached to the choice $[\omega]$) to be 
	\begin{equation}\label{eq:omega pi}
		\Omega_{\pi,\chi_{\infty}} \defeq w^{\bj}(\det w_n) \cdot \cP_{\chi_\infty}([\omega])^{-1},
	\end{equation}
	recalling $w_n$ is the longest Weyl element in $\GL_n$ and $w^{\bj}$ is from \S\ref{sec:p-adic Hecke}.
\end{definition}

\begin{remark}
	The simplified (purely archimedean) statements above suffice for our purposes; but in fact, \cite{JST} prove considerably more. Using global input, one may find a $\Q(\pi,\eta)$-rational structure on $\h^t(\fg,K_\infty^\circ; \pi_\infty \otimes V_\lambda^\vee(\C))$ (see Lemma 4.7 \emph{op.\ cit}.), and one may take $[\omega]$ to be $\Q(\pi,\eta)$-rational. Such a rational $[\omega]$, and the resulting period $\Omega_{\pi,\chi_\infty}$, depends globally on $\pi$ (not just $\pi_\infty$) and is well-defined up to $\Q(\pi,\eta)^\times$-multiple.
\end{remark}

\begin{remark}\label{rem:period relations}
	The sign $w^{\bj}(\det w_n)$ looks a little artificial, and could in any case be absorbed into the choice of $\kappa_{\bj}$. It is an artifact of our later normalisation of branching laws, and we include it to allow for cleaner interpolation statements in later sections. More precisely, in Theorem \ref{thm:critical value} it will combine with the local terms $\chi_{\pri}(\det w_n)$ at $\pri|p$ (from Proposition \ref{prop:zeta p}), ultimately yielding a factor of $\chi_{[p]}(\det w_n)$ in the final interpolation formula. This factor can then be normalised away with a Dirac distribution; see \S\ref{sec:proof theorem A}.
	
	Once $[\omega]$ is fixed, $\Omega_{\pi,\chi_{\infty}}$ depends only on the sign $\epsilon_{\chi} \in \{\pm 1\}^{\Sigma(\R)}$ (notation as in \cite[\S2.2.1]{BW_CJM}), and $\bj$, through the choice of branching law $\kappa_{\bj}$. We further expect that for certain normalised choices of $\kappa_{\bj}$, that $\Omega_{\pi,\chi_{\infty}}$ should have very light, explicit dependence on $\bj$.
	
	To be more precise: fix $\chi_0$ finite order, and let $\chi_j \defeq \chi_0|\cdot|^j$. In \cite{JST} the $\kappa_{\bj}$ are normalised (in Lemma 4.8) to send $\smallmatrd{1}{-w_n}{w_n}{1} \cdot v_0^\vee$ to 1, where $v_0^\vee$ is a highest weight vector in $V_\lambda^\vee$. One of their main technical results is then that there exists $\epsilon \in \{\pm1\}$ such that $\epsilon^j \cdot i^{-jn[F:\Q]} \cdot \cP_{\chi_{j,\infty}}$ is independent of $j$. This leads to their proof of the period relations for $L(\pi\times\chi_0,s)$ as $s$ ranges over critical integers $j$. 
	
	We will later, via a different method, specify choices of $\kappa_{\bj}$ via a different method that allows $p$-adic interpolation. Naturally one should expect that our choices are related to those of \cite{JST}, and that such a relation should imply the value of our $p$-adic $L$-function at any special value \emph{on the cyclotomic line} would satisfy exactly the interpolation predicted by Panchishkin (including the factor at infinity). If $F$ has a real embedding, this accounts for all special values.	When $F$ has no real embedding, it is natural to hope that an adaptation of the methods in \cite{JST} would provide an analogous result for the more general $\cP_{\chi_\infty}$, which would mean our $p$-adic $L$-functions have the correct interpolation factor at $\infty$ at \emph{all} special values.
\end{remark}

\subsection{Summary}

We collect together our choices.

\begin{test-data}\label{test-data}
	 Let $\pi$ be a RASCAR of weight $\lambda$ spherical at every $\pri|p$. Then:
	\begin{itemize}\s
\item For each $\pri|p$, fix a regular spin $Q$-refinement $\tilde\pi_{\pri} = (\pi_{\pri},\alpha_{\pri})$ of $\pi_{\pri}$. Write $\tilde\pi = (\pi, \alpha)$ for the collection of these choices, for $\alpha = (\alpha_{\pri})_{\pri|p}$. Let $\varphi_{\pri}^{\alpha_{\pri}} \in \pi_{\pri}$ be as in Proposition \ref{prop:zeta p}.

\item Away from $p$, let $\varphi_{f}^{\mathrm{FJ},(p)} = \otimes_{v\nmid p\infty}\varphi_v^{\mathrm{FJ}} \in \pi_f^{(p)}$, for $\varphi_v^{\mathrm{FJ}}$ as in Proposition \ref{prop:FJ test vector}.

\item At $\infty$, let $[\omega] \in \h^t(\fg,K_\infty^\circ; \pi_\infty \otimes V_\lambda^\vee(\C))$ be the class from Theorem \ref{thm:non-vanishing}.
\end{itemize}
Write $\varphi_f^{\mathrm{FJ},\alpha} = \varphi_f^{\mathrm{FJ},(p)} \otimes (\otimes_{\pri|p}\varphi_{\pri}^{\alpha_{\pri}})$.
\end{test-data}

Let $u^{-1}t_p^{\beta} = \prod_{\pri|p}u^{-1}t_{\pri}^{\beta_{\pri}} \in G(\Qp)$. Then for any $\chi \in \mathrm{Crit}_p(\pi)$, we have computed $\cP_{\chi_\infty}([\omega]) \cdot Z_f^\circ(u^{-1}t_p^\beta\cdot \varphi_f^{\mathrm{FJ},\alpha}, \chi, 1/2)$ in terms of $L(\pi\times\chi,1/2)$, explicit terms at $p$, and a \emph{non-zero} period.


\section{Critical $L$-values via cohomology}\label{sec:critical L-values}

We now describe a cohomological interpretation of the Friedberg--Jacquet integrals, and hence the standard critical $L$-values, following \cite{GR2} and \cite{JST}. We begin to tailor our approach towards $p$-adic interpolation, however, with features from \cite{DJR18,BDW20,BDGJW}.

Throughout, let $\pi$ be a RASCAR with an $(\eta,\psi)$-Shalika model. Let $\lambda$ be the weight of $\pi$, with purity weight $\usw$. Let $\chi \in \mathrm{Crit}_p(\pi)$ be an algebraic Hecke character of infinity type $\bj$ and conductor $p^{\beta_0}= \prod_{\pri|p}\pri^{\beta_{0,\pri}}$, where $\beta_0 = (\beta_{0,\pri}) \in \Z_{\geq 0}^{\pri|p}$ is a multiexponent. For technical reasons (cf.\ Proposition \ref{prop:zeta p}), we let $\beta_{\pri} = \max(1,\beta_{0,\pri})$ and let $\beta = (\beta_{\pri}) \in \Z_{\geq 1}^{\pri|p}$.

\subsection{Automorphic cycles}

One of the main results of \cite[\S2]{FJ93}, reinterpreted in \cite[Prop.\ 4.1]{JST}, is that the product 
\[
	Z_\mu^\circ \cdot Z_f^\circ(-,\chi,s) : (\wedge^t \fX_H^{\C})^* \otimes \pi \otimes V_{\chi,s-1/2}^H \longrightarrow \C,
\]
introduced here in \S\ref{sec:local}, can be computed  via period integrals for $H \subset G$. In particular, they showed that there exists a Haar measure $\mu$ such that
 \[
 	[Z_\mu^\circ(-,\chi_\infty,s) \cdot Z_f^\circ(-,\chi_f,s)\big](\omega \otimes \varphi \otimes 1) \cdot L(\pi\times\chi) = \int_{X_\beta} \varphi \cdot \omega,
 \]
 where $X_\beta$ is an \emph{automorphic cycle} for $H$ that we define below. We will explain this, and interpret the right-hand side via cohomology.
 
For the rest of \S\ref{sec:critical L-values}, let $K \subset G(\A_f)$ be an open compact subgroup fixing $\varphi_f^{\mathrm{FJ},\alpha}$ from Test Data \ref{test-data}. Recall $L_\infty = H(\R) \cap K_\infty$. (Here we take intersections with $H$ with respect to $\iota$).

\begin{definition}\label{def:auto cycle}
Let $L \subset H(\A_f) \cap K$ be open compact.	The \emph{automorphic cycle of level $L$} is
	\[
	X_L \defeq H(\Q)\backslash H(\A) /L L_\infty^\circ.
	\]
\end{definition}

By \cite[Lemma 2.7]{Ash80}, $\iota : H \hookrightarrow G$ induces a proper map $X_L \hookrightarrow S_K$, which we also denote $\iota$.

In \S\ref{sec:evaluations}, we will choose $L_\beta \subset H(\A_f)$ at $\beta$, which (as in \cite[\S4.1]{BDW20}) is sufficiently small that:
\begin{itemize}\s
	\item $\chi$ and $\eta$ are trivial on $L_\beta$,
	\item $X_\beta \defeq X_{L_\beta}$ is a real manifold,
	\item and $H(\Q) \cap hL_\beta L^\circ_\infty h^{-1} = Z_G(\Q) \cap L_\beta L_\infty^\circ$ for all $h \in H(\A)$. 
\end{itemize}

Recall $t = \mathrm{dim}_{\R}(\cX_H) = \mathrm{dim}_{\R}(X_\beta)$, and $F$ has $r$ (resp.\ $s$) real (resp.\ complex) places.

\begin{lemma} 
	We have $t = \mathrm{dim}_{\R}(X_\beta) = r(n^2+n-1) + s(2n^2-1).$
\end{lemma}
\begin{proof}
	At each real place, we get a contribution of $[\GL_n(\R)\times \GL_n(\R)]/[\mathrm{O}_n(\R)\times \mathrm{O}_n(\R)]\R^\times$, which has dimension $2(n^2) - 2(n^2-n)/2 - 1 = n^2 + n - 1$. 
	
	At each complex place, we get $[\GL_n(\C)\times\GL_n(\C)]/[\mathrm{U}_n(\C) \times \mathrm{U}_n(\C)]\C^\times$, which has dimension $2(2n^2) - [2(n^2) + 2 - 1] = 2n^2-1$, noting $\mathrm{dim}_{\R}(\C^\times \cap [\mathrm{U}_n(\C)\times \mathrm{U}_n(\C)]) = 1$.
\end{proof}

\begin{remark}
	Recall the supported range of cuspidal cohomology degrees 
	\[rn^2 + s(2n^2-n) \leq i \leq r(n^2+n-1) + s(2n^2+n-1)\]
	 from \eqref{eq:cuspidal range}. We see that at each real place, there is a contribution of $n^2+n-1$ to the dimension of automorphic cycles, exactly at the top of the range $[n^2, n^2+n-1]$. For complex places, the contribution of $2n^2-1$ is just below the middle of the range $[2n^2-n, 2n^2+n-1]$. 
\end{remark}

\subsection{Evaluation maps, I}\label{sec:evaluations JST}
Still summarising \cite{JST}, we now give a global version. Recall $\kappa_{\bj} : V_\lambda^\vee \to V_{\bj,-\usw-\bj}^H$ from Lemma \ref{lem:branching law}. Let $\cEv_{\bj}^{\mathrm{JST}}$ be the composition
\begin{align}\label{eq:evaluation JST}
	\cEv_{\bj}^{\mathrm{JST}} : \hc{t}(S_K,\cV_\lambda^\vee(\C)) &\otimes \h^0(X_\beta,\cV_{-\bj,\usw+\bj}^H(\C))\\
	  &\xrightarrow{ \  \iota^* \otimes \mathrm{id} \ }\hc{t}(X_\beta,\cV_{\lambda}^\vee(\C)) \otimes \h^0(X_\beta,\cV_{-\bj,\usw+\bj}^H(\C))\notag\\
	  	 &\xrightarrow{ \ \kappa_{\bj}  \otimes \mathrm{id} \ }\hc{t}(X_\beta,\cV_{\bj,-\usw-\bj}^H(\C)) \otimes \h^0(X_\beta,\cV_{-\bj,\usw+\bj}^H(\C))\notag\\
	& \xrightarrow{ \ \cup \ } \hc{t}(X_\beta,\C) \isorightarrow \C.\notag
\notag\end{align}
The final arrow is an integration map $\int_{X_\beta}$ (that we will make precise in \eqref{eq:integration}). 

As in \cite[(4.12)]{JST}, the Lie algebra cohomology yields a natural injective map (and class)
\begin{align*}
	\Theta_\chi : V_{\chi_f}^H \defeq V_{\chi_f,0}^H \longhookrightarrow \h^0(X_\beta,&\cV_{-\bj,\usw+\bj}^H(\C)),\\
	&[V_{\chi_f}^H] \defeq \Theta_\chi(1).
\end{align*}
The choice of $[\omega]$ (from Theorem \ref{thm:non-vanishing}) yields, via \eqref{eq:automorphic cohomology}, an injective map 
\begin{equation}\label{eq:Theta 1}
\Theta_{[\omega]}^K : \pi_f^K \longhookrightarrow \hc{t}(S_K,\cV_\lambda^\vee(\C)).
\end{equation}
The following is the commutative diagram after \cite[Lem.\ 4.11]{JST}. The term $\mathrm{vol}(L_\beta)$ arises because of our normalisation of $\int_{X_\beta}$, compared to their period integral in (4.7) \emph{op.\ cit}.

\begin{proposition}
	We have a commutative diagram
	\[
	\xymatrix@C=30mm{
	\pi_f^K \otimes V_{\chi_f} \ar[d]^{\Theta_{[\omega]}^K\otimes \Theta_\chi} \ar[r]^-{ Z_f^\circ\big(-,\chi_f,1/2\big)}	 & \C\ar[d]^{\frac{w^{\bj}(\det w_n)}{\mathrm{vol}(L_\beta)} \cdot \frac{L\big(\pi\times\chi, 1/2\big)}{\Omega_{\pi,\chi_\infty}}}\\
	\hc{t}(S_K,\cV_\lambda^\vee(\C)) \otimes \h^0(X_\beta,\cV_{-\bj,\usw+\bj}^H(\C)) \ar[r]^-{\cEv_{\bj}^{\mathrm{JST}}} & \C.
	}
\]
\end{proposition}

Finally, we rephrase in language we will consider in the next sections, and summarise all the above results. Let
\begin{equation}\label{eq:Ev chi 1}
	\cEv_{\chi}^{\mathrm{JST}} : \hc{t}(S_K,\cV_\lambda^\vee(\C)) \longrightarrow \C, \qquad \phi \longmapsto \cEv_{\bj}^{\mathrm{JST}}\big(\phi \otimes [V_{\chi_f}^H]\big).
\end{equation}
From Proposition \ref{prop:FJ}, we see immediately that:
\begin{corollary}\label{cor:critical L-value}
If $\varphi_f = \varphi_f^{\mathrm{FJ},(p)} \otimes (\otimes_{\pri|p}\varphi_{\pri}) \in \pi_f^K$, then for all $\chi \in \mathrm{Crit}(\pi)$ with infinity type $\bj$, we have
\begin{equation}
	\cEv_{\chi}^{\mathrm{JST}}\circ \Theta_{[\omega]}^K(\varphi_f) = \frac{w^{\bj}(\det w_n) }{\mathrm{vol}(L_\beta)} \cdot \prod_{v\nmid p\infty}\chi(\varpi_v^{\delta_v})^{-n}\cdot \prod_{\pri|p}Z_{\pri}\big(\varphi_{\pri},\chi_{\pri},\tfrac12\big) \cdot \frac{L^{(p)}\big(\pi\times\chi,\tfrac12\big)}{\Omega_{\pi,\chi_{\infty}}}.
\end{equation}
\end{corollary}

Here the term $w^{\bj}(\det w_n)$ is coming from \eqref{eq:omega pi}.


\section{Evaluation maps}\label{sec:evaluations}

We now give abstract generalisations of $\cEv_{\chi}^{\mathrm{JST}} $. To connect the above to the study in \cite{BDW20}, we first give a more pedestrian reinterpretation of $\cEv_{\chi}^{\mathrm{JST}} $, before describing analogues with \emph{$p$-adic} local systems. As in the last section, $\pi$ will be a RASCAR with an $(\eta,\psi)$-Shalika model, of weight $\lambda$, with purity weight $\usw$. Let $\chi$ be an algebraic Hecke character of conductor $p^{\beta_0}$, with $\beta_0 = (\beta_{0,\pri})_{\pri|p} \in \Z^{\pri|p}_{\geq 0}$, and infinity type $\bj$ critical for $\pi$. Again set $\beta_{\pri} = \max(1,\beta_{0,\pri})$ and let $\beta = (\beta_{\pri}) \in \Z_{\geq 1}^{\pri|p}$.

We want to work with $p$-adic, rather than complex, coefficients; so recall we fixed an isomorphism $i_p : \C \to \overline{\Q}_p$. Via this isomorphism, we identify $V_\lambda^\vee(\C)$ and $V_\lambda^\vee(\overline{\Q}_p)$, and let
\begin{equation}\label{eq:Theta p}
	\Theta_{[\omega]}^{K,i_p} = i_p \circ \Theta_{[\omega]}^K : \pi_f^K \longhookrightarrow \hc{t}(S_K,\cV_\lambda^\vee(\overline{\Q}_p))
\end{equation}
for $\Theta_{[\omega]}^K$ the map from \eqref{eq:Theta 1}.

\subsection{Automorphic cycles revisited}
\subsubsection{Level structures}
We now specify the levels $L_\beta$. Let $K = \prod_v K_v \subset G(\A_f)$, such that:
\begin{itemize}\s
\item[(i)] For $v \nmid p\infty$, we take $K_v$ fixing $\varphi_v^{\mathrm{FJ}}$ (from \S\ref{sec:FJ finite});
\item[(ii)] For $\pri|p$, we take $K_{\pri} = J_{\pri}$ parahoric (as in \S\ref{sec:notation}).
\end{itemize}

Let $u \in G(\A_f)$ be the matrix with $u_{v} = 1$ if $v\nmid p$, and $u_{\pri} = \smallmatrd{1}{w_n}{0}{1}$ for all $\pri|p$. For a multiexponent $\beta = (\beta_{\pri})_{\pri|p}$, let $L_\beta = L^{(p)} \prod_{\pri|p} L_{\pri}^{\beta_{\pri}} \subset H(\A_f)$, where (taking all intersections with $H$ with respect to $\iota$):
\begin{itemize}\s
	\item[(i)] away from $p$, $L^{(p)} \subset G(\A_f^{(p)})$ is the principal congruence subgroup of some (fixed, suppressed) prime-to-$p$ ideal $\m \subset \cO_F$, chosen so that $L^{(p)} \subset H(\A_f^{(p)}) \cap K^{(p)}$; 
	\item[(ii)] and $L_{\pri}^{\beta_{\pri}} = H(\Zp) \cap K_{\pri} \cap \Big(u_{\pri}^{-1}t_{\pri}^{\beta_{\pri}} \cdot K_{\pri} \cdot t_{\pri}^{-\beta_{\pri}}u_{\pri} \Big).$
\end{itemize}
We take $\m$ large enough that the three conditions after Definition \ref{def:auto cycle} are satisfied, and as \emph{op.\ cit.}, we let $X_\beta \defeq X_{L_\beta}$, the \emph{automorphic cycle of level $p^\beta$}.

\begin{remark}
	The matrix $u_{\pri}$ is an open-orbit representative of the spherical variety $G/H$ (that is, $B(F_{\pri}) \cdot u_{\pri} \cdot H(F_{\pri})$ is open in $G(F_{\pri})$). The matrix $t_{\pri}$ induces the action of the $U_{\pri}$ Hecke operator.
\end{remark}

\subsubsection{Connected components and fundamental classes}

\begin{lemma}\label{lem:1 mod p}
	The map $(\det_1/\det_2, 1/\det_2)$ descends to a surjective map
	\[
	L_\beta \longtwoheadrightarrow \big(1 + p^\beta\m\widehat{\cO}_F\big) \times \big(1 + \m\widehat{\cO}_F\big) \subset \A_F^\times \times \A_F^\times.
	\]
\end{lemma}

\begin{proof}
	Identical to \cite[Lem.\ 2.1]{DJR18} and \cite[Lem.\ 4.5]{BDGJW}.
\end{proof}

By strong approximation for $H$, the connected components of $X_\beta$ are indexed by 
\begin{equation}\label{eq:component group}
	\pi_0(X_\beta) \defeq \cl_F^+(p^\beta \m) \times \cl_F^+( \m),
\end{equation}
where $\cl_F^+(I)$ is the narrow ray class group of conductor $I$.  For $\delta \in H(\A_f)$, we write $[\delta]$ for its associated class in $\pi_0(X_\beta)$ and denote the corresponding connected component
\[
X_\beta[\delta] \defeq H(\Q)\backslash H(\Q)\delta L_\beta H(\R)^\circ/L_\beta L_\infty^\circ.
\]

We can describe these components as arithmetic quotients of the symmetric space $\cX_H = H(\R)^\circ/L_\infty^\circ$ from \S\ref{sec:local zeta infinity}. For $\beta \in \Z_{\geq 1}$ and $\delta \in H(\A_f)$, define a congruence subgroup
\begin{equation}\label{eq:gamma beta delta}
	\Gamma_{\beta,\delta} \defeq H(\Q) \cap \delta L_\beta H(\R)^\circ \delta^{-1}.
\end{equation} 
The group $\Gamma_{\beta,\delta}$  acts on $\cX_H$ by left translation via $\Gamma_{\beta,\delta} \hookrightarrow H(\R)^\circ$, and there is an isomorphism
\[
c_\delta : \Gamma_{\beta,\delta}\backslash \cX_H \isorightarrow X_\beta[\delta] \subset X_\beta, \ \ \ [h_\infty]_\delta \mapsto [\delta h_\infty],
\]
where if $[h_\infty] \in \cX_H$, we write $[h_\infty]_\delta$ for its image in $\Gamma_{\beta,\delta}\backslash \cX_H$.

Since $\Gamma_{\beta,\delta}\backslash \cX_H$ is a connected orientable manifold of dimension $t$, there is an isomorphism
\[
 \h_t^{\mathrm{BM}}(\Gamma_{\beta,\delta}\backslash \cX_H,\Z) \cong \Z.
\]
A \emph{fundamental class} $\theta_\delta \in  \h_t^{\mathrm{BM}}(\Gamma_{\beta,\delta}\backslash \cX_H,\Z)$ is a generator of this group; choosing such a class is equivalent to choosing an orientation on $\Gamma_{\beta,\delta}\backslash \cX_H$. Following the process described in \cite[\S2.2.5]{DJR18} and \cite[\S4.2.3]{BDW20}, we choose a family $(\theta_\delta)_\delta$ of such classes compatibly as $\delta$ (and $\beta$) varies. Choose, once and for all, an ordered basis of the tangent space of the symmetric space $H_\infty^\circ/L_\infty^\circ$, which induces -- for any $\beta$ -- a fundamental class $\theta[\beta]$ for $X_\beta[1] \subset X_\beta$, the connected component of the identity. For any $\delta$, let $\theta_{[\delta]} = \delta_*\theta[\beta]$, which is a fundamental class for $X_\beta[\delta]$. Then we finally define $\theta_\delta = c_\delta^*\theta_{[\delta]}$. 

For each $\delta$, the choice of $\theta_\delta$ yields an isomorphism
\[
(- \cap \theta_\delta) : \hc{t}(\Gamma_{\beta,\delta}\backslash \cX_H,\Z) \isorightarrow \Z.
\]

Our choices, and the maps $c_\delta$, induce an integration map promised in \eqref{eq:evaluation JST}:
\begin{equation}\label{eq:integration}
	\int_{X_\beta} \defeq \sum_\delta (-\cap \theta_\delta) \circ c_\delta^* : \hc{t}(X_\beta,\Z) \longmapsto \Z
\end{equation}

\subsection{The map $\cEv_{\chi}^{\mathrm{JST}} $ revisited}

Now let $\Qp\subset L \subset \Qpbar$ containing $\Q(\pi,\eta)$. For $p$-adic interpolation we must twist the map $\iota : X_\beta \to S_K$. Define
\[
	\iota_\beta : X_\beta \longrightarrow S_K, \qquad [h] \mapsto [hu^{-1}t_p^{\beta}], \qquad h \in H(\A), \ \ t_p = \prod_{\pri|p}t_{\pri}^{\beta_{\pri}}.
\]
One may check $\iota_\beta^*\cV_\lambda^\vee = \iota^*\cV_\lambda^\vee$, so for $\delta \in H(\A)$, we have maps
\[
\hc{t}(S_K,\cV_\lambda^\vee(L)) \xrightarrow{ \ \iota_\beta^* \ } \hc{t}(X_\beta,\iota^*\cV_\lambda^\vee(L)) \xrightarrow{ \ c_\delta^* \ } \hc{t}(\Gamma_{\beta,\delta}\backslash \cX_H, c_\delta^*\iota^*\cV_{\lambda}^\vee(L)),
\]
where $c_\delta^*\iota^*\cV_\lambda^\vee$ can be checked to be the local system given by locally constant sections of $\Gamma_{\beta,\delta}\backslash [\cX_H \times V_\lambda^\vee(L)] \to \Gamma_{\beta,\delta}\backslash \cX_H$ with action $\gamma([h_\infty],v) = ([\gamma h_\infty], \gamma \cdot v)$ (cf.\ \S\ref{sec:local systems}). 

If $M$ is a left $\Gamma_{\beta,\delta}$-module, let  $M_{\Gamma_{\beta,\delta}} \defeq M/\{m - \gamma \cdot m : m \in M, \gamma \in \Gamma_{\beta,\delta}\}$ be the coinvariants. The quotient map $V_\lambda^\vee(L) \to (V_{\lambda}^\vee(L))_{\Gamma_{\beta,\delta}}$ trivialises the local system, inducing a map
	\[
		\mathrm{coinv}_{\beta,\delta} : \hc{t}(\Gamma_{\beta,\delta}\backslash \cX_H, c_\delta^*\iota^*\cV_{\lambda}^\vee(L)) \longrightarrow \hc{t}(\Gamma_{\beta,\delta}\backslash \cX_H, \Z) \otimes_{\Z} (V_{\lambda}^\vee(L))_{\Gamma_{\beta,\delta}}.
	\]

\begin{definition}\label{def:arch evaluation map}
 Define
	\begin{align}\label{eq:evaluation definition 2}
		\cEv_{\beta,\delta}   : \htc(S_K,\cV_\lambda^\vee(L)) & \xrightarrow{\ c_\delta^* \circ (\iota_\beta)^*\ } \hc{t}(\Gamma_{\beta,\delta}\backslash\cX_H,c_\delta^*\iota^*\cV_\lambda^\vee(L))\\
		& \xrightarrow{\mathrm{coinv}_{\beta,\delta}}  \hc{t}(\Gamma_{\beta,\delta}\backslash\cX_H, \Z)\otimes_{\Z} (V_\lambda^\vee(L))_{\Gamma_{\beta,\delta}}\\
		& \labelisorightarrow{ - \cap \theta_\delta} (V_\lambda^\vee(L))_{\Gamma_{\beta,\delta}}. \notag
	\end{align}
\end{definition}

Recall $\VH$ is the algebraic $H$-representation of highest weight $(\bj,-\usw-\bj)$. Recalling $w_p^{\bj}$ from \eqref{eq:w^j}, $\VH(L)$ is an $H(\Qp)$-representation, with $h =(h_1,h_2) \in H(\Qp)$ acting by 
\begin{equation}\label{eq:wjw}
w_p^{\bj}(\det(h_1)) w_p^{-\usw-\bj}(\det(h_2)) \defeqrev (w_p^{\bj} \otimes w_p^{-\usw-\bj})(h).
\end{equation}
Recall we chose (via Lemma \ref{lem:branching law}) a non-zero element $\kappa_{\bj} \in \Hom_{H(\Qp)}(V_\lambda^\vee(L),\VH(L))$.

\begin{lemma}\label{lem:coinvariants}
The map $\kappa_{\bj}$ induces a map $\kappa_{\bj} :  (V_\lambda^\vee)_{\Gamma_{\beta,\delta}} \to \VH(L)$.
\end{lemma}
\begin{proof}
	It suffices to prove $\Gamma_{\beta,\delta}$ acts trivially on the target. If $\gamma = (\gamma_1,\gamma_2) \in \Gamma_{\beta,\delta}$, then 
	\[
	\det(\gamma_1),\det(\gamma_2) \in \cO_{F,+}^\times.
	\]
	 Moreover $\det(\gamma_1/\gamma_2) \equiv 1 \newmod{p^\beta}$  and $\det(\gamma_2) \equiv 1 \newmod{\m}$ (by Lemma \ref{lem:1 mod p}).

By assumption there exists an algebraic Hecke character $\chi$ of infinity type $\bj$ and conductor dividing $p^\beta$; this forces $w_p^{\bj}(\det(\gamma_1/\gamma_2)) = w^{\bj}(\det(\gamma_1/\gamma_2)) = 1$, as in \cite[\S5.2.3]{BW_CJM}. Similarly, existence of $\eta$ with infinity type $\usw$ and conductor dividing $\m$ forces $w_p^{-\usw}(\det(\gamma_2)) = 1$. Thus $(w_p^{\bj}\otimes w_p^{-\usw-\bj})(\gamma) = 1$, i.e.\ $\gamma$ acts trivially on $\VH(L)$, as required.
\end{proof}

\begin{definition}\label{def:ev_chi}
	Let	\begin{align*}
		\cEv_{\chi} : \hc{t}(S_K,\cV_\lambda^\vee(L)) &\longrightarrow L(\chi), \\
		\phi &\longmapsto \sum_{\delta\in\pi_0(X_\beta)} (\chi\otimes\chi^{-1}\eta^{-1})(\delta) \cdot \Big[\kappa_{\bj}\circ \cEv_{\beta,\delta}(\phi)\Big].
	\end{align*}
	Here we have chosen a basis $u_{\bj}$ of $V_{\bj,-\usw-\bj}^H(L)$ to identify it with $L(\chi)$. Via $i_p$, we can do this compatibly with the choices in \S\ref{sec:evaluations JST}.
\end{definition}

By considering how all these maps behave on the local systems, we see that actually this `new' construction is just a twisted version of the map $\cEv_{\chi}^{\mathrm{JST}}$ from \eqref{eq:Ev chi 1} (hence it is independent of the choices of $\delta$). Precisely, recall 
\[	
	u^{-1}t_p^\beta = \textstyle\prod_{\pri|p}u^{-1}t_{\pri}^{\beta_{\pri}} \subset G(\Qp), \qquad \text{ and let } \qquad K_\beta \defeq u^{-1}t_p^\beta K t_p^{-\beta} u.
\]
 Note $L_\beta \subset K_\beta$, so we can apply $\cEv_{\chi}^{\mathrm{JST}}$ at level $K_\beta$. If $\varphi_f \in \pi^K$, then $u^{-1}t_p^\beta \cdot \varphi_f \in \pi_f^{K_\beta}$, and:

\begin{lemma}\label{lem:Ev 1 vs Ev 2}
	The following diagram commutes:
	\[
		\xymatrix@C=30mm{
			\pi_f^K \ar[r]^-{\Theta_{[\omega]}^{K,i_p}}\ar[d]_{\varphi_f \mapsto u^{-1}t_p^\beta \cdot \varphi_f} & \hc{t}(S_K,\cV_\lambda^\vee(\Qpbar)) \ar[r]^-{\cEv_\chi} & \Qpbar\ar[d]^{i_p^{-1}}\\
			\pi_f^{K_\beta} \ar[r]^-{\Theta_{[\omega]}^{K_\beta}} & \hc{t}(S_{K_\beta},\cV_\lambda^\vee(\C)) \ar[r]^-{\cEv_\chi^{\mathrm{JST}}} & \C.
		}
	\]
\end{lemma}

Note that the twisting matrix $u^{-1}t_p^\beta$ is exactly what appeared in Proposition \ref{prop:zeta p}.


\subsection{Abstract evaluation maps} \label{sec:abstract evaluations}

So far we have only considered evaluation maps only with coefficients in $\cV_\lambda^\vee$, the local system attached to $V_\lambda^\vee$ with its $G(\Q)$-action. Ultimately we will consider coefficients in $p$-adic distributions, for which only the $K$-local systems make sense. We now present a version of the above maps for $K$-local systems, with abstract coefficient modules, generalising those constructed in \cite[\S4]{BDW20} (when $F$ is totally real) and \cite[\S10.1]{BW_CJM} (for $\GL_2$).

The constructions/proofs of \cite{BDW20} all go through exactly as \emph{op.\ cit}., so we omit details.

Recall $\Delta_p \subset G(\Qp)$ from \S\ref{sec:*-action}. Let $M$ be a left $\Delta_p$-module, with action denoted $\bullet$. (When $M = \cD_\lambda$, this will be the $*$-action. When $M = V_\lambda^\vee$, recall from \S\ref{sec:actions} that we have two such actions, the $\cdot$- and $*$-actions, and we shall need to consider both). 

The level $K$ acts on $M$ via its projection to $J_p = \prod_{\pri|p}J_{\pri} \subset \Delta_p$, giving a local system $\sM$ on $S_K$ via \S\ref{sec:local systems}. If $\gamma \in \Gamma_{\beta,\delta}$ (from \eqref{eq:gamma beta delta}), then $\delta^{-1}\gamma\delta \in L_\beta \subset K$, so $\Gamma_{\beta,\delta}$ acts on $M$ via $\gamma \bullet_{\Gamma_{\beta,\delta}}m \defeq (\delta^{-1}\gamma\delta)_f\bullet m$. Hence we may take coinvariants for this action.

The evaluation maps for $\sM$ are inspired by those in Definition \ref{def:arch evaluation map} for $\cV_\lambda^\vee$. A crucial difference is that $\iota_\beta^*\sM$ is \emph{not} $\iota^*\sM$ any more, but there is a map
\[
\tau_\beta^\bullet : \iota_\beta^*\sM \to \iota^*\sM, \qquad (h,m) \mapsto (h,u^{-1}t_p^\beta\bullet m).
\]

\begin{definition}\label{def:evaluation map}
	The \emph{evaluation map for $(M,\bullet)$ of level $p^\beta$ at $\delta$} is the composition
	\begin{align}\label{eq:evaluation definition}
		\sEv_{\beta,\delta}^{M, [\bullet]}   : \htc(S_K,\sM)  \xrightarrow{\ \tau_{\beta}^{\bullet} \circ (\iota_\beta)^*\ } \hc{t}(X_\beta,&\iota^*\sM) \xrightarrow{c_\delta^*} \hc{t}(\Gamma_{\beta,\delta}\backslash \cX_H, c_\delta^*\iota^*\sM)\\
		& \xrightarrow{\coinv_{\beta,\delta}}  \hc{t}(\Gamma_{\beta,\delta}\backslash\cX_H, \Z) \otimes M_{\Gamma_{\beta,\delta}} \labelisorightarrow{ - \cap \theta_\delta} M_{\Gamma_{\beta,\delta}}. \notag
	\end{align}
\end{definition}

We track dependence on $M,\delta,\beta$. The following are proved exactly as in \cite[\S4.3]{BDW20}.

\begin{lemma}(Variation in $M$; \cite[Lem.\ 4.6]{BDW20}) \label{lem:pushforward}
	Let $\kappa : M \rightarrow N$ be a $\Delta_p$-module map. There is a commutative diagram
	\[
	\xymatrix@C=18mm@R=6mm{
		\hc{t}(S_K, \sM) \ar[r]^-{\sEv_{\beta,\delta}^{M, [\bullet]}} \ar[d]^-{\kappa_*} & M_{\Gamma_{\beta,\delta}}\ar[d]^-{\kappa}\\
		\hc{t}(S_K,\sN) \ar[r]^-{\sEv_{\beta,\delta}^{N, [\bullet]}} & N_{\Gamma_{\beta,\delta}}.
	}
	\]
\end{lemma}

\begin{proposition}(Variation in $\delta$; \cite[Prop.\ 4.9]{BDW20}) \label{prop:ind of delta}
	Let $N$ be a left $H(\A)$-module, with action $\bullet$, such that $H(\Q)$ and $H(\R)^\circ$ act trivially. Let $\kappa : M\to N$ be a map of $L_\beta$-modules. Then
	\[
	\sEv_{\beta,[\delta]}^{M,[\bullet],\kappa} \defeq \delta \bullet \left[\kappa \circ \sEv_{\beta,\delta}^{M, [\bullet]}\right] : \hc{t}(S_K,\sM) \longrightarrow N
	\]
	is well-defined and independent of the representative $\delta$ of $[\delta]$.
\end{proposition}

Fix $\pri|p$, and define $\beta' = (\beta_\mathfrak{q}')_{\mathfrak{q}|p}$, where $\beta_{\pri}' = \beta_{\pri} + 1$ and $\beta_\mathfrak{q}' = \beta_\mathfrak{q}$ for $\mathfrak{q} \neq \pri$. We have natural projections  
\[
\mathrm{pr}_{\beta,\pri} : X_{\beta'} \longrightarrow X_\beta, \ \ \ \mathrm{pr}_{\beta,\pri} : \pi_0(X_{\beta'}) \rightarrow \pi_0(X_\beta).
\]
The action of $t_{\pri} \in \Delta_p$ on $M$ yields an action of $U_{\pri}^\bullet \defeq [Kt_{\pri}K]$ on $\hc{t}(S_K,\sM)$, via \cite[\S2.3.2]{BDW20}.

\begin{proposition}(Variation in $\beta$; \cite[Prop.\ 4.10]{BDW20}) \label{prop:evaluations changing beta} Let $N$ and $\kappa$ be as in Proposition~\ref{prop:ind of delta}. If $\beta > 0$, then as maps $\hc{t}(S_K,\sM) \to N$ we have
	\[
	\sum_{[\eta] \in \mathrm{pr}_{\beta,\pri}^{-1}([\delta])} \sEv_{\beta',[\eta]}^{M,[\bullet],\kappa}   =
	\sEv_{\beta,[\delta]}^{M,[\bullet],\kappa} \circ U_{\pri}^\bullet.
	\]
\end{proposition}


\subsection{Classical evaluation maps}\label{sec:classical evaluations}

We now use the above to construct linear functionals
\[
	\sEv_{\chi}^{[\bullet]} : \hc{t}(S_K,\sV_\lambda^\vee(L)) \longrightarrow L(\chi),
\]
analogous to $\cEv_\chi$ with $\cV_\lambda^\vee$. We have
\[
\sE_{\beta,\delta}^{V_\lambda^\vee,[\bullet]} : \hc{t}(S_K,\sV_\lambda^\vee(L)) \longrightarrow (V_\lambda^\vee(L))_{\Gamma_{\beta,\delta}}.
\]
Lemma \ref{lem:coinvariants} says $\kappa_{\bj}$ factors through the $\Gamma_{\beta,\delta}$-coinvariants. That lemma used the $\cdot$-action on $V_\lambda^\vee$, but by the same proof it is also true of the $*$-action.

To apply Proposition \ref{prop:ind of delta} to $\kappa_{\bj}$, we now extend $\VH(L)$ to an $H(\A)$-module in such a way that $H(\Q)$ and $H(\R)^\circ$ act trivially (as required by the statement of that proposition).

\begin{definition}
	Let $V_{\chi,[p]}^H(L)$ be the 1-dimensional $H(\A)$-module $\chi_{[p]}\otimes\chi_{[p]}^{-1}\eta_{[p]}^{-1}$, recalling  $\chi_{[p]}$ and $\eta_{[p]}$ are the ray class characters attached to $\chi$ and $\eta$ in \S\ref{sec:p-adic Hecke}. Precisely, it is the space $L(\chi)$, with $h = (h_1,h_2) \in H(\A)$ acting as
	\[
	h\cdot v = \Big[\chi_{[p]}\left(\mathrm{det}\left(\tfrac{h_1}{h_2}\right)\right) \cdot \eta_{[p]}\left(\mathrm{det}\left(\tfrac{1}{h_2}\right)\right)\Big] v.
	\] 

\end{definition}

By construction, $H(\Q)$ and $H(\R)^\circ$ act trivially on $V_{\chi,[p]}^H$.

\begin{lemma}\label{lem:Vjw to Vvarphi}
	The identity map is an isomorphism of 1-dimensional $L_\beta$-modules
	\[
		\VH(L) \isorightarrow V_{\chi,[p]}^H(L) \cong L(\chi).
	\]
	Here $L_\beta$ acts on $\VH$ via projection to $H(\Qp)$, and on $V_{\chi,[p]}^H$ by restricting the $H(\A)$-action.
\end{lemma}
\begin{proof}
Via \eqref{eq:wjw},  $\ell = (\ell_1,\ell_2) \in L_\beta$ acts on $\VH(L)$ as
\begin{equation}\label{eq:L_beta 1}
\ell \cdot v = (w_p^{\bj}\otimes w_p^{-\usw-\bj})(\ell_p) \   v.
\end{equation}
 By definition of $\chi_{[p]}$ and $\eta_{[p]}$, the action on $V_{\chi,[p]}^{H}$ is by
\begin{equation}\label{eq:L_beta 2}
	\ell \cdot v = \chi(\ell_1\ell_2^{-1}) \eta(\ell_2^{-1})\times  (w_p^{\bj}\otimes w_p^{-\usw-\bj})(\ell_p)  \ v.
\end{equation}
Now $\det(\ell_1\ell_2^{-1}) \equiv 1\newmod{p^\beta\widehat{\cO}_F}$ by Lemma \ref{lem:1 mod p}, so as $\chi$ has conductor dividing $p^\beta$, we see $\chi(\ell_1\ell_2^{-1}) \defeq \chi(\det(\ell_1\ell_2^{-1})) = 1$. Similarly $\det(\ell_2) \equiv 1 \newmod{\m}$, and the conductor of $\eta$ divides $\m$, we have $\eta(\ell_2^{-1}) = 1$. In particular, \eqref{eq:L_beta 1} and \eqref{eq:L_beta 2} agree, proving the lemma.
\end{proof}

\begin{corollary}
	For any $\delta \in H(\A_f)$, the map
\[
\sEv_{\chi, [\delta]}^{[\bullet]} : \hc{t}(S_K,\sV_\lambda^\vee(L)) \xrightarrow{\ \sEv_{\beta,\delta}^{V_\lambda^\vee,[\bullet]} \ } \big(V_{\lambda}^\vee(L)\big)_{\Gamma_{\beta,\delta}} \xrightarrow{\kappa_{\bj}} \VH(L) \isorightarrow V_{\chi,[p]}^H \xrightarrow{\ \cdot \delta \ } V_{\chi,[p]}^H \cong L(\chi)
\]
is well-defined and depends only on the class $[\delta] \in \pi_0(X_\beta)$.
\end{corollary}
\begin{proof}
	Immediate by combining Proposition \ref{prop:ind of delta} and Lemma \ref{lem:Vjw to Vvarphi}.
\end{proof}

\begin{definition}
	The \emph{classical evaluation map attached to $\chi$ and the action $\bullet$} is the map
		\[
			\sEv_{\chi}^{[\bullet]} \defeq \sum_{\delta \in \pi_0(X_\beta)} \sEv_{\chi,[\delta]}^{[\bullet]} : \hc{t}(S_K,\sV_\lambda^\vee(L)) \longrightarrow L.
		\]
	It depends on the choices of $\kappa_{\bj}$ and a basis $u_{\bj}$ of the line $V_{\chi,[p]}^H \cong L(\chi)$, each unique up to scalar.
\end{definition}

As mentioned previously, there are two natural choices of action $\bullet$ on $V_\lambda^\vee(L)$: the standard $\cdot$-action (which connects more cleanly to the theory from \S\ref{sec:critical L-values}), and the $*$-action inherited as a quotient of $\cD_\lambda$. We get two evaluation maps $\sEv_{\chi}^{[\cdot]}$ and $\sEv_{\chi}^{[*]}$. Since the only dependence on $\bullet$ is in the map $\tau_\beta^\bullet$, from \eqref{eq:dot vs *} we have (recalling $\chi$ has conductor related to $p^\beta$)
\begin{equation}\label{eq:ev dot vs *}
	\sEv_{\chi}^{[*]}  = \lambda(t_p^\beta) \times \sEv_\chi^{[\cdot]}.
\end{equation}

\subsection{Comparison between $\cEv_\chi$ and $\sEv_\chi^{[\cdot]}$}
We now connect the evaluation with $\bullet = \cdot$ to Definition \ref{def:ev_chi} and hence to \S\ref{sec:critical L-values}. Let $\phi \in \hc{t}(S_K,\cV_\lambda^{\vee}(L))$, and let $\upsilon$ denote the natural isomorphism
\[
\upsilon : \hc{t}(S_K,\cV_\lambda^\vee(L)) \isorightarrow \hc{t}(S_K,\sV_\lambda^\vee(L))
\]
from \S\ref{sec:local systems}. Then, taking $\cdot$-actions everywhere:

\begin{proposition}
We have $\cEv_{\chi}(\phi) = \sEv_{\chi}^{[\cdot]}(\upsilon(\phi)).$
\end{proposition}

\begin{proof}
We have a commutative diagram (cf.\ \cite[Prop.\ 4.6]{DJR18} and \cite[Prop.\ 4.1]{BDJ17})
\[
\xymatrix@R=3mm@C=35mm{
\phi \sar{d}{\in}\ar@{|->}[r] & \upsilon(\phi)\sar{d}{\in}\\
\hc{t}(S_K,\cV_\lambda^\vee(L))\ar[dd]^{\iota_\beta^*} 
\ar[r]^{(g,v) \longmapsto (g,g_p^{-1}\cdot v)} & \hc{t}(S_K,\sV_\lambda^\vee(L))\ar[dd]^{\tau_\beta^\cdot \circ \iota_\beta^*}
\\
&\\
\hc{t}(X_\beta,\iota^*\cV_\lambda^\vee(L))\ar[r]^{(h,v) \longmapsto (h,h_p^{-1}\cdot v)} \ar[dd]^{c_\delta^*}
& \hc{t}(X_\beta,\iota^*\sV_\lambda^\vee(L))\ar[dd]^{c_\delta^*}
\\
&\\
\hc{t}(\Gamma_{\beta,\delta}\backslash\cX_H, c_\delta^*\iota^*\cV_\lambda^\vee(L)) \ar[r]^{([h_\infty],v) \mapsto ([h_\infty], \delta_p^{-1}\cdot v)}\ar[dd]^{(-\cap\theta_{\delta}) \circ \mathrm{coinv}_{\beta,\delta} } & \hc{t}(\Gamma_{\beta,\delta}\backslash\cX_H, c_\delta^*\iota^*\sV_\lambda^\vee(L))\ar[dd]^{(-\cap\theta_{\delta}) \circ \mathrm{coinv}_{\beta,\delta} }\\
&\\
(V_\lambda^\vee(L))_{\Gamma_{\beta,\delta}} \ar[r]^{v \longmapsto \delta_p^{-1}\cdot v} & (V_\lambda^\vee(L))_{\Gamma_{\beta,\delta}},
}
\]
where the horizontal maps are induced by the stated maps of local systems. In particular,
\[
\delta_p^{-1} \cdot \cEv_{\beta,\delta}(\phi) = \sEv_{\beta,\delta}^{[\cdot]}(\upsilon(\phi))
\]
as elements of $(V_\lambda^\vee(L))_{\Gamma_{\beta,\delta}}$. Composing with the $H(\Qp)$-module map $\kappa_{\bj}$ gives
\begin{equation}\label{eq:local system prop}
(w_p^{-\bj}\otimes w_p^{\usw+\bj})(\delta_p)\Big[\kappa_{\bj}\circ \cEv_{\beta,\delta}(\phi)\Big] = \Big[\kappa_{\bj}\circ \sEv_{\beta,\delta}^{[\cdot]}(\upsilon(\phi))\Big].
\end{equation}
Now consider both sides as elements of $V_{\chi,[p]}^H$ by Lemma \ref{lem:Vjw to Vvarphi}, and act by $\delta$ on both sides (via the $H(\A)$ action on $V_{\chi,[p]}^H$). On the left-hand side, the factor $(w_p^{\bj}\otimes w_p^{-\usw-\bj})(\delta_p)$ in $\chi_{[p]}\otimes\chi_{[p]}^{-1}\eta_{[p]}^{-1}$ cancels with the left-most term in \eqref{eq:local system prop}, so  we have
\[
\delta\cdot\left((w_p^{-\bj}\otimes w_p^{\usw+\bj})(\delta_p) \Big[\kappa_{\bj}\circ \cEv_{\beta,\delta}(\phi)\Big]\right) =(\chi\otimes\chi^{-1}\eta^{-1})(\delta) \cdot \Big[ \kappa_{\bj} \circ \cEv_{\beta,\delta}(\phi)\Big].
\]
On the right-hand side we get, by definition, $\sEv_{\chi,[\delta]}^{[\cdot]}(\upsilon(\phi)).$ Summing both sides over $\delta \in \pi_0(X_\beta)$ completes the proof.
\end{proof}

The following theorem, now using the $*$-action, summarises the last three sections. As in the last line of the proof of \cite[Thm.\ 4.7]{DJR18}), the \emph{global Gauss sum} is
\begin{equation}\label{eq:tau chi f}
	\tau(\chi_f) \defeq \prod_{\chi_v \text{ unramified}}\chi_v(\varpi_v)^{-\delta_v} \cdot \prod_{\chi_v \text{ ramified}} \tau(\chi_v);
\end{equation}
this is different from some treatments, as our $\psi$ has conductor $\fd^{-1}$, rather than $\widehat{\cO}_F$.

\begin{theorem}\label{thm:critical value}
Let $\pi$ and $\varphi_f^{\mathrm{FJ},\alpha}$ be as in Test Data \ref{test-data}.  For all $\chi \in \mathrm{Crit}(\pi)$, we have
	\begin{align}
		i_p^{-1}\circ \sEv_{\chi}^{[*]}\circ \Theta_{[\omega]}^{K,i_p}(\varphi_f^{\mathrm{FJ},\alpha}) =& A \cdot \chi_{[p]}(\det w_n) \cdot \lambda(t_p^\beta)  \\
		&\times  \tau(\chi_f)^n\cdot  \prod_{\pri|p} Q'(\pi_{\pri},\chi_{\pri}) \cdot \frac{L^{(p)}\big(\pi\times\chi,\tfrac12\big)}{\Omega_{\pi,\chi_\infty}},\notag
	\end{align}
where 
	\[
	Q'(\pi_{\pri},\chi_{\pri})=\left\{\begin{array}{cl}  q_{\pri}^{\beta_{\pri}\left(\tfrac{n^2-n}{2}\right)}  &: \chi_{\pri} \text{ ramified},\\
		 \displaystyle  \alpha_{\pri} \cdot   \prod_{i=n+1}^{2n}
			\frac{1-\UPS_{\pri,i}^{-1}\chi_{\pri}^{-1}(\varpi_{\pri})q_{\pri}^{-1/2}}{1-\UPS_{\pri,i}\chi_{\pri}(\varpi_{\pri})q_{\pri}^{-1/2}} &: \chi_{\pri} \text{ unramified},\end{array}\right.
	\]
	and 
\begin{equation}\label{eq:A}
A \defeq \frac{1}{\mathrm{vol}(L_1)} \cdot \prod_{\pri|p} \left(\frac{q_{\pri}^n}{(q_{\pri}-1)^n} \cdot q_{\pri}^{\delta_{\pri}(n^2-n)/2}\right)
\end{equation}
is a constant independent of $\chi$. Here $L_1 \defeq L_{(1,...,1)}$ (i.e.\ we take $\beta_{\pri} = 1$ for all $\pri$), and we consider $\det(w_n)$ as an element of $(\cO_F\otimes \Zp)^\times$.
\end{theorem}

\begin{proof}
By \eqref{eq:ev dot vs *} (to pass from $*$ to $\cdot$ evaluations, introducing $\lambda(t_p^\beta)$), Corollary \ref{cor:critical L-value} and Lemma \ref{lem:Ev 1 vs Ev 2}, noting that $u^{-1}t_p^\beta$ is trivial outside primes above $p$, we have
	\begin{align*}
	i_p^{-1}\circ \sEv_{\chi}^{[*]}\circ \Theta_{[\omega]}^{K,i_p}(\varphi_f^{\mathrm{FJ},\alpha}) = \lambda(t_p^\beta)& \cdot  \frac{ w^{\bj}(\det w_n)}{\mathrm{vol}(L_\beta)} \cdot \prod_{v\nmid p\infty}\chi_v(\varpi_v)^{-n\delta_v}\\
	&\times \prod_{\pri|p}Z_{\pri}\big([u^{-1}t_{\pri}^{\beta_{\pri}}\cdot\varphi_{\pri}^{\alpha_{\pri}}],\chi_{\pri},\tfrac12\big) \cdot \frac{L^{(p)}\big(\pi\times\chi,\tfrac12\big)}{\Omega_{\pi,\chi_\infty}}.
\end{align*}
By \cite[Lem.\ 4.4]{BDGJW} (cf.\ Lemma \ref{lem:1 mod p}) we have 
\begin{equation}\label{eq:vol Lbeta}\textstyle
	\mathrm{vol}(L_\beta) = \mathrm{vol}(L_1) \cdot \delta_B(t_p^\beta) = \mathrm{vol}(L_1) \cdot \prod_{\pri|p}q_{\pri}^{-\beta_{\pri}n^2}.
\end{equation}
The twisted local zeta integrals at $p$ were evaluated\footnote{This is where we want to use $\beta \in \Z_{\geq 1}^{\pri|p}$, rather than the `true' conductor $p^{\beta_0}$; for $p$-adic interpolation it is necessary to twist non-trivially at each $\pri$, even when the conductor $\beta_{0,\pri}$ itself is trivial.} in Proposition \ref{prop:zeta p}. In the ramified case, the term $q_{\pri}^{-\beta_{\pri}n^2}$ is the difference between $Q$ (from that proposition) and $Q'$ (here); in the unramified case, it cancels with the power of $q_{\pri}^{-n^2}$ that appears in $Q$. Via \S\ref{sec:p-adic Hecke}, the local terms involving $\det w_n$ combine to give $\chi_{[p]}(\det w_n)$. Combining all of this, we get the factor $Q'(\pi_{\pri},\alpha_{\pri})$ and $\chi_{\pri}(\varpi_{\pri}^{-\delta_{\pri}})$ (resp.\ $\tau(\chi_{\pri})$) if $\chi_{\pri}$ is unramified (resp.\ ramified). Finally we conclude by the identity \eqref{eq:tau chi f}.
\end{proof}


\section{$p$-adic interpolation of evaluation maps}\label{sec:interpolate evaluations}

We will interpolate the $L$-values appearing on the right-hand side of Theorem \ref{thm:critical value} by interpolating the left-hand side, i.e.\ $\sEv_{\chi}^{[*]}$, as $\chi$ varies. The main result of \S\ref{sec:interpolate evaluations} is Proposition \ref{prop:evaluations interpolate}.

\subsection{Alignment of branching laws}

The evaluation maps of the previous section depended on choices of bases 
\begin{equation}\label{eq:kappa u}
	\kappa_{\bj} \in \Hom_{H(\Qp)}(V_\lambda^\vee(L), \VH(L)), \qquad u_{\bj} \in \VH(L) \cong L.
\end{equation}
The choices can be combined into a single choice $\kappa_{\bj}^\circ : V_\lambda^\vee(L) \to L$ defined by 
\begin{equation}\label{eq:kappa circ}
	\kappa_{\bj}(\mu) = \kappa_{\bj}^\circ(\mu) \cdot u_{\bj} \qquad \forall \mu \in V_{\lambda}^\vee(L).
\end{equation}

In Theorem \ref{thm:critical value}, these choices manifested themselves on the left-hand side in the definition of $\sEv_{\chi}^{[*]}$, and on the right-hand side (via $i_p$) in the zeta integral at infinity. To interpolate the $\sEv_{\chi}^{[*]}$ requires a careful alignment of the choices as $\bj$ varies, which we carry out here. The main idea is that we can collapse all the different choices (as $\bj$ varies) of branching law $\kappa_{\bj}$ for $H \subset G$ onto a \emph{single} choice of branching law for $G_n \defeq \mathrm{Res}_{F/Q}(\GL_n) \subset H$, diagonally embedded. 

Write $\lambda = (\lambda',\lambda'')$, where $\lambda',\lambda''$ are two weights for $G_n$. Then as $H$-representations $V_\lambda^H \cong V_{\lambda'}^{G_n} \otimes V_{\lambda''}^{G_n}$. We can restrict this under the diagonal embedding of $G_n$, obtaining:

\begin{lemma}
	The restriction $V_{\lambda}^H(L)|_{G_n}$ to a diagonal copy of $G_n$ contains the $G_n$-representation $w_p^{\usw}\circ \det$ with multiplicity 1.
\end{lemma}
\begin{proof}
For each $\sigma \in \Sigma$, the individual weight $\lambda_\sigma \in \Z^{2n}$ is pure by \S\ref{sec:weights}; that is, we have $\lambda_{\sigma,i} + \lambda_{\sigma,2n+1-i} = \sw_\sigma$ for $1 \leq i \leq n$. Combining over all $i$, this means that
\[
	\lambda''_\sigma = (\lambda_{\sigma,n+1},\dots, \lambda_{\sigma,2n}) = (-\lambda_{\sigma,n} + \sw_\sigma, \dots, -\lambda_{\sigma,1} + \sw_\sigma) = (\lambda'_\sigma)^\vee + (\sw_\sigma,...,\sw_\sigma),
\]
noting that $(\lambda_\sigma')^\vee_i = -\lambda_{\sigma, n+1-i}$. Thus $\lambda'' = (\lambda')^\vee + (\usw,\cdots, \usw)$. It follows that
	\[
V_{\lambda'}^{G_n} \otimes V_{\lambda''}^{G_n} \cong V_{\lambda'}^{G_n} \otimes (V_{\lambda'}^{G_n})^\vee \otimes \mathrm{det}^{\usw},
	\]
	which contains $\det^{\usw}$ with multiplicity 1. 
\end{proof}

Fix a generator $v^H_\lambda \in w_p^{\usw}\circ \det \subset V_{\lambda}^H(L)|_{G_n}$. Recall the description $V_\lambda = \Ind_{\overline{Q}}^G V_\lambda^H$ from \eqref{eq:double induction}. Let 
\begin{equation}\label{eq:N_Q^times}
	N_Q^\times(\zp) \defeq \left\{\smallmatrd{1}{X}{0}{1} \in N_Q(\zp) : X \in G_n(\Zp)\right\}.
\end{equation}

\begin{proposition}\label{prop:explicit v_j}
	\begin{enumerate}[(i)] \s
		\item	For each $\bj$ critical for $\lambda$, there exists a unique
		\[
		\big[v_{\bj} : G(\Zp) \to V_\lambda^H(L) \big] \in V_\lambda(L)
		\]
		with
		\[
		v_{\bj}\left[\smallmatrd{1}{X}{0}{1}\right] = w_p^{\bj}(\det(w_nX)) \cdot \Big(\left\langle \smallmatrd{X}{}{}{1}\right\rangle_\lambda \cdot v_\lambda^H\Big)
		\]
		for all $\smallmatrd{1}{X}{0}{1} \in N_Q^\times(\Zp)$. 	
		
		\item $v_{\bj}$ generates $V^H_{-\bj,\underline{\sw}+\bj}(L) \subset V_\lambda(L)|_{H(\Qp)}.$
	\end{enumerate}
\end{proposition}

\begin{proof}
	This follows \cite[Lem.\ 5.11]{BDW20}. Take \emph{some} choices of $\kappa_{\bj}$ and $u_{\bj}$. These define dual bases
	\begin{equation}\label{eq:kappa u}
		\kappa_{\bj}^\vee \in \Hom_{H(\Qp)}(V_{-\bj,\usw+\bj}^H(L), V_\lambda(L)), \qquad u_{\bj}^\vee \in V^H_{-\bj,\usw+\bj}(L).
	\end{equation} 
Thus $\kappa_{\bj}^\vee(u_{\bj}^\vee) \in V_\lambda(L)$; and via \eqref{eq:double induction}, we may define
\[
	v_{\lambda,\bj}^H \defeq \kappa_{\bj}^\vee(u_{\bj}^\vee)\left[\smallmatrd{1_n}{1_n}{0}{1_n}\right] \in V_\lambda^H(L).
\]
Arguing exactly as in \cite[Lem.\ 5.8]{BDW20}, we see $v_{\lambda,\bj}^H$ is a non-zero element of $w_p^{\usw}\circ \det \subset V_{\lambda}^H(L)|_{G_n}$, and hence -- up to rescaling $\kappa_{\bj}$ -- we may assume $v_{\lambda,\bj}^H = v_{\lambda}^H$, independent of $\bj$.  

For these rescaled choices, define $v_{\bj} \defeq \kappa_{\bj}^\vee(u_{\bj}^\vee)$. Then (i) follows exactly as in \cite[Lem.\ 5.7]{BDW20}, with uniqueness following from Zariski-density of $\overline{N}_QHN_Q^\times(\Zp) \subset G(\Zp)$ (as in \cite[Lem.\ 5.11(i)]{BDW20}). Part (ii) is then identical to \cite[Lem.\ 5.11(ii)]{BDW20}.
\end{proof}

\begin{definition}
		We fix the choice of $\kappa_{\bj}^\circ : V_{\lambda}^\vee(L) \to L$ by setting $\kappa_{\bj}^\circ(\mu) = \mu(v_{\bj})$. Note this corresponds to the choices of $\kappa_{\bj}$ and $u_{\bj}$ in the proof, after rescaling so that the attached $v_{\lambda,\bj}^H = v_{\lambda}^H$.
\end{definition}

\subsection{Branching laws for distributions}

We now construct a `master branching law' $\kappa_\lambda$, interpolating all the $\kappa_{\bj}$ above. Let $\cA(\Galp,L)$ be the space of locally analytic functions on $\Galp$, with an $H(\A)$-action by
\[
(h_1,h_2) * f(x) \defeq \eta_{[p]}(h_2) f(\det(h_1^{-1}h_2) x).
\]
This induces a dual left-action on $\cD(\Galp,L) \defeq \mathrm{Hom}_{\mathrm{cts}}(\cA(\Galp,L),L)$. Recall if $\chi \in \mathrm{Crit}_p(\pi)$, then $\chi_{[p]}$ induces a character on $\Galp$ via \S\ref{sec:p-adic Hecke}.  The following is immediate.

\begin{lemma}\label{lem:D action}
\begin{enumerate}[(i)]\s
\item $H(\R)^\circ$ and $H(\Q)$ act trivially on $\cD(\Galp,L)$.
\item The map $\mu \mapsto \mu(\chi_{[p]})$ defines an $H(\A)$-module map $\cD(\Galp,L) \to V_{\chi,[p]}^H(L)$.
\end{enumerate}
\end{lemma}

Recall from \eqref{eq:cl SES} there is a natural map $(\cO_F\otimes \Zp)^\times \to \Galp$, which we denote $\jmath$.  Given $f \in \cA(\Galp,L)$ and $x \in (\cO_F\otimes\Zp)^\times$, abusing notation we write $f(x) \defeq f(\jmath(x))$. For such $f$, define a map
\begin{align}\label{eq:v_lambda}
v_\lambda(f) : N_Q^\times(\Zp) &\longrightarrow V_\lambda^H(L),\\
\smallmatrd{1}{X}{0}{1} &\longmapsto f(\det(w_n X)) \left(\langle\smallmatrd{X}{}{}{1}\rangle_\lambda \cdot v_\lambda^H\right),\notag
\end{align}
noting $\det(X) \in (\cO_F\otimes\Zp)^\times$ by the definition  \eqref{eq:N_Q^times} of $N_Q^\times(\Zp)$. Extending by 0, $v_\lambda(f)$ is a locally analytic function $N_Q(\Zp) \to V_\lambda^H(L)$, whence it defines an element $v_\lambda(f) \in \cA_\lambda$ via the parahoric transformation law \eqref{eq:parahoric transform}.

If there was a function $w_p^{\bj} \in \cA(\Galp,L)$, then comparing Proposition \ref{prop:explicit v_j} with \eqref{eq:v_lambda}, we would formally have $v_\lambda(w_p^{\bj})(n) = v_{\bj}(n)$ for  $n \in N_Q^\times(\Zp)$. Unfortunately the function $w_p^{\bj}$ on $(F\otimes \Qp)^\times$ does \emph{not} induce a function on $\Galp$ in general (since it is not $\Q^\times$-invariant). However, if $\chi \in \mathrm{Crit}_p(\pi)$ with infinity type $\bj$ and conductor dividing $p^\beta$, then $\chi_{[p]}$ \emph{is} a function on $\Galp$, and $\chi_{[p]}(x) = w_p^{\bj}(x)$ for $x \in \A_F^\times$ with $x \equiv 1\newmod{p^\beta}$. As such, recalling that $\beta = (\beta_{\pri}) \in \Z_{\geq 1}^{\pri|p}$, let:
\begin{itemize}\s
\item $\sU_\beta = \jmath(1+p^\beta\cO_F\otimes\Zp) \subset \Galp$,
\item $\cA(\sU_\beta,L) \subset \cA(\Galp,L)$ be the space of functions supported on $\sU_\beta$,
\item $N^\beta_Q(\Zp) \defeq \{n \in N_Q(\Zp) : n \equiv \smallmatrd{1}{w_n}{0}{1_n} \newmod{p^\beta}\} \subset N_Q^\times(\Zp),$
\item  $J_p^\beta \defeq J_p \cap \overline{N}_Q(\Zp)H(\Zp)N_Q^\beta(\Zp) \subset J_p$,
\item $\cA_\lambda^\beta \subset \cA_\lambda$ (resp.\ $\cD_\lambda^\beta \subset \cD_\lambda$) be the space of functions (resp.\ distributions) supported on $J_p^\beta$.
\end{itemize}
If $n = \smallmatrd{1}{X}{0}{1} \in N_Q^\beta(\Zp)$, then $X \equiv w_n \newmod{p^\beta}$, so $\det(w_nX) \equiv 1 \newmod{p^\beta}$. If $f \in \cA(\sU_\beta,L)$, we can thus define a function $v_\lambda^\beta(f) : \N_Q^\beta(\Zp) \to V_\lambda^H(L)$ exactly as in \eqref{eq:v_lambda}. After extending by 0 to $N_Q(\Zp)$, as above we  obtain a function
\[
v_\lambda^\beta : \cA(\sU_\beta,L) \longrightarrow \cA_\lambda^\beta. 
\]
Then, recalling that $\chi$ has conductor dividing $p^\beta$ and infinity type $\bj$:
\begin{lemma}
If $g \in J_p^\beta$, then $v_\lambda^\beta(\chi_{[p]})(g) = v_{\bj}(g)$.
\end{lemma}
\begin{proof}
As both $v_\lambda^\beta(\chi_{[p]})$ and $v_{\bj}$ are functions $J_p \to V_\lambda^H(L)$ satisfying the same parahoric transformation law, it suffices to show this for $g = \smallmatrd{1}{X}{0}{1} \in N_Q^\beta(\Zp)$. Then $\det(w_nX) \equiv 1\newmod{p^\beta}$, so $\chi_{[p]}(\det(w_nX)) = w_p^{\bj}(\det(w_nX))$. Thus when we plug $\chi_{[p]}$ into \eqref{eq:v_lambda}, we recover the formula for $v_{\bj}$ in Proposition \ref{prop:explicit v_j}.
\end{proof}

Dualising $v_\lambda^\beta$ gives a map $\kappa_\lambda^\beta : \cD_\lambda^\beta \to \cD(\sU_\beta,L) \subset \cD(\Galp,L)$. 

\begin{proposition}\label{prop:master branching law}
If $\chi \in \mathrm{Crit}_p(\pi)$ has conductor dividing $p^\beta$ and infinity type $\bj$, we have a commutative diagram
\[
\xymatrix@C=20mm{\cD_\lambda^\beta(L) \ar[d]^{r_\lambda}\ar[r]^-{\kappa_{\lambda}^\beta} & \cD(\Galp,L) \ar[d]^{\mu \mapsto \mu(\chi_{[p]})}\\
V_\lambda^\vee(L) \ar[r]^-{\kappa_{\bj}^\circ} & L.
}
\]
\end{proposition}

\begin{proof}
Let $\mu \in \cD_\lambda^\beta$. Then
\[
\kappa_\lambda^\beta(\mu)(\chi_{[p]}) = \int_{\Galp} \chi_{[p]} \cdot d\kappa_\lambda^\beta(\mu) = \int_{\Galp}v_\lambda^\beta(\chi_{[p]}) \cdot d\mu = \int_{J_p^\beta}v_\lambda^\beta(\chi_{[p]}) \cdot d\mu = \int_{J_p^\beta}v_{\bj} \cdot d\mu = \kappa_{\bj}^\circ \circ r_\lambda(\mu),
\]
where we have used that $\mu$ has support on $J_p^\beta$,  and $\mu(v_{\bj}) = \mu(r_\lambda(v_{\bj}))$ since  $v_{\bj} \in V_\lambda$. 
\end{proof}

\subsection{Overconvergent evaluation maps}

Recall $K$ (hence $L_\beta$) acts on $\cD_\lambda$ via \S\ref{sec:*-action} by projection to $K_p \subset J_p$, giving a local system $\sD_\lambda$ on $S_K$ by \S\ref{sec:local systems}. 

\begin{lemma}
	The action of $L_\beta \subset K$ preserves $\cD_\lambda^\beta$.
\end{lemma}
\begin{proof}
	If $f \in \cA_\lambda^\beta$, and $\ell = (\ell_1,\ell_2) \in L_\beta$, then for any $\smallmatrd{1}{X}{0}{1} \in N_Q(\Zp)$, we have $(\ell * f)\smallmatrd{1}{X}{0}{1} = \langle \ell \rangle_H \cdot f\smallmatrd{1}{\ell_1^{-1}X\ell_2}{0}{1}$. By the proof of \cite[Lem.\ 4.5]{BDGJW}, we have $\ell_2 \equiv w_n\ell_1w_n \newmod{p^\beta}$; so $\ell_1^{-1}X\ell_2 \equiv w_n \newmod{p^\beta}$ if and only if $X \equiv \ell_2w_n\ell_2^{-1} \equiv w_n \newmod{p^\beta}$. Thus $\ell * f$ has support on $N_Q^\beta(\Zp)$ if $f$ does; so $L_\beta$ preserves $\cA_\lambda^\beta$, hence $\cD_\lambda^\beta$.
\end{proof}

To make Proposition \ref{prop:master branching law} useful, we need the following support condition.

\begin{lemma}\label{lem:support}
The map 
\[
\sEv_{\beta,\delta}^{\cD_\lambda,[*]} : \hc{t}(S_K,\sD_\lambda) \to (\cD_\lambda)_{\Gamma_{\beta,\delta}}
\]
has image in $(\cD_\lambda^\beta)_{\Gamma_{\beta,\delta}} \subset (\cD_\lambda)_{\Gamma_{\beta,\delta}}$.
\end{lemma}
\begin{proof}
Identical to \cite[Lem.\ 12.4]{BDGJW}.
\end{proof}

The following, and Lemma \ref{lem:D action}, allows us to use Proposition \ref{prop:ind of delta}.

\begin{lemma}\label{lem:kappa equivariant}
	The map $\kappa_\lambda^\beta : \cD_\lambda^\beta \to \cD(\Galp,L)$ is a map of $L_\beta$-modules.
\end{lemma}
\begin{proof}
It suffices to prove that $v_\lambda^\beta : \cA(\sU_\beta,L) \to \cA_\lambda^\beta$ is a map of $L_\beta$-modules.  This is almost identical to \cite[Prop.\ 6.11(ii)]{BDW20}. We observe that if $f \in \cA(\Galp,L)$ and $\ell = (\ell_1,\ell_2) \in L_\beta$, then $(\ell * f)(x) = w_p(\det(\ell_{2,p}))^{\usw}f(\det(\ell_{1,p}^{-1}\ell_{2,p}))$, since $\eta(\det(\ell_2)) = 1$ (as in Lemma \ref{lem:coinvariants}) and the prime-to-$p$ part of $\ell_1^{-1}\ell_2$ lies in $\sU(p^\infty)$, i.e.\ acts trivially on $\cl_F^+(p^\infty)$.  The rest of the proof proceeds exactly as \emph{op.\ cit}.
\end{proof}

\begin{proposition}
	For any $\delta \in H(\A_f)$, the map
\[
\sEv_{\beta,[\delta]}^{\dagger} : \hc{t}(S_K,\sD_\lambda^\vee(L)) \xrightarrow{\ \sEv_{\beta,\delta}^{\cD_\lambda,[*]} \ } \big(\cD_{\lambda}^\beta(L)\big)_{\Gamma_{\beta,\delta}} \xrightarrow{\kappa_{\lambda}^\beta} \cD(\Galp,L) \xrightarrow{\ \cdot \delta \ } \cD(\Galp,L)
\]
is well-defined and depends only on the class $[\delta] \in \pi_0(X_\beta)$.
\end{proposition}
\begin{proof}
	Immediate by combining Proposition \ref{prop:ind of delta} with Lemmas \ref{lem:D action}(i) and \ref{lem:kappa equivariant}.
\end{proof}

\begin{definition}
	The \emph{overconvergent evaluation map of level $\beta$} is the map
		\[
			\sEv_{\beta}^{\dagger} \defeq \sum_{\delta \in \pi_0(X_\beta)} \sEv_{\beta,[\delta]}^{\dagger} : \hc{t}(S_K,\sD_\lambda^\vee(L)) \longrightarrow \cD(\Galp,L).
		\]
\end{definition}

\subsection{Interpolation of classical evaluation maps}
The map $r_\lambda$ from \eqref{eq:r_lambda} is by definition a $\Delta_p$-module map when $V_\lambda^\vee$ is given the $*$-action, hence induces a $U_p^*$-equivariant map
\begin{equation}\label{eq:specialisation}
	r_\lambda : \hc{t}(S_K,\sD_\lambda) \longrightarrow \hc{t}(S_K,\sV_\lambda^\vee).
\end{equation}

\begin{proposition}\label{prop:evaluations interpolate}
	For any $\chi$ of conductor dividing $p^\beta$, we have a commutative diagram
	\[
	\xymatrix@C=25mm{
	\hc{t}(S_K,\sD_\lambda(L)) \ar[r]^-{\sEvs_{\beta}^\dagger} \ar[d]^{r_\lambda} & \cD(\Galp,L)\ar[d]^{\mu \mapsto \mu(\chi_{[p]})} \\
	\hc{t}(S_K,\sV_\lambda^\vee(L)) \ar[r]^-{\sEvs_{\chi}^{[*]}} & L.	
	}
\]
\end{proposition}

\begin{proof}
It suffices to check every square in 
	\[
\xymatrix@C=15mm{
	\hc{t}(S_K,\sD_\lambda(L)) \ar[r]^-{\sEvs_{\beta,\delta}^{\cD_\lambda,[*]}} \ar[d]^{r_\lambda} & (\cD_\lambda^\beta(L))_{\Gamma_{\beta,\delta}} \ar[r]^-{\kappa_{\lambda}}\ar[d]^{r_\lambda} &\cD(\Galp,L)\ar[d]^{\mu \mapsto \mu(\chi_{[p]})}\ar[r]^{\delta *} &\cD(\Galp,L)\ar[d]^{\mu \mapsto \mu(\chi_{[p]})} \\
	\hc{t}(S_K,\sV_\lambda^\vee(L)) \ar[r]^-{\sEvs_{\beta,\delta}^{V_\lambda^\vee,[*]}} &  (V_\lambda^\vee(L))_{\Gamma_{\beta,\delta}} \ar[r]^-{\kappa_{\bj}^\circ} & L \ar[r]^{\delta * } & L
}
\]
commutes, since the top row is $\sEv^\dagger_\beta$ and the bottom row $\sE_{\chi}^{[*]}$. The left-most square commutes by Lemma \ref{lem:pushforward} (noting the top arrow is well-defined by Lemma \ref{lem:support}). The middle square commutes by Proposition \ref{prop:master branching law}. The right-most square commutes by Lemma \ref{lem:D action}(ii), noting that $\delta$ acts on the bottom row via the identification $L \cong V_{\chi,[p]}^H(L)$.
\end{proof}

\section{The $p$-adic $L$-function}

\subsection{Construction}\label{sec:construction}
Let now $\pi$ be a RASCAR of weight $\lambda$, spherical at all $\pri|p$, and let $\tilde\pi = (\pi, \{\alpha_{\pri}\}_{\pri}\}$ be a regular spin $Q$-refinement as in \S\ref{sec:local at p}. Let $\varphi_{f}^{\mathrm{FJ},\alpha} \in \pi_f$ and $[\omega]$ be chosen as in Test Data \ref{test-data}, and let 
\begin{equation}\label{eq:phi}
	\phi_{\tilde\pi} \defeq \Theta_{[\omega]}^{K,i_p}(\varphi_f^{\mathrm{FJ},\alpha}) \in \hc{t}(S_K,\sV_\lambda^\vee(\Qpbar)).
\end{equation}
There exists some finite extension $L/\Qp$ such that $\phi_{\tilde\pi}$ is defined over $L$, rather than $\Qpbar$. The map $\Theta_{[\omega]}^{K,i_p}$ is Hecke-equivariant when we give the right-hand side the $\cdot$-action, hence 
\[
U_{\pri}^\cdot \phi_{\tilde\pi}  = \alpha_{\pri}\phi_{\tilde\pi}, \qquad U_{\pri}^*\phi_{\tilde\pi} = \lambda(t_{\pri}) \alpha_{\pri} \phi_{\tilde\pi},
\]
the second via \eqref{eq:dot vs *}. Let $\alpha_{\pri}^\circ \defeq \lambda(t_{\pri})\alpha_{\pri}$, the $U_{\pri}^*$-eigenvalue.

\begin{definition}\label{def:non-critical}
	Recall the map $r_\lambda$ from \eqref{eq:specialisation}. We say  $\tilde\pi$ is \emph{non-$Q$-critical} if $r_{\lambda}$ becomes an isomorphism after restricting to the $\tilde\pi$-generalised eigenspaces for the Hecke algebra (precisely, the algebra $\cH$ from \cite[Def.\ 2.10]{BDW20}).	We say $\tilde\pi$ has \emph{non-$Q$-critical slope} if 
	\[
		e_{\pri} \cdot v_p(\alpha_{\pri}^\circ) < \mathrm{min}_{\sigma \in \Sigma(\pri)}(1 + \lambda_{\sigma,n} - \lambda_{\sigma,n+1}) \qquad \text{ for all $\pri|p$},
	\]
	where $\Sigma(\pri)$ is the set of $\sigma : F \hookrightarrow \C \isorightarrow \overline{\Q}_p$ that extend to $F_{\pri} \hookrightarrow \Qpbar$, and $e_{\pri}$ is the ramification index of $\pri|p$. 
	(If $\tilde\pi$ is \emph{ordinary}  -- i.e.\ $v_p(\alpha_{\pri}^\circ) = 0$ for all $\pri$ -- then it has non-$Q$-critical slope.)
\end{definition} 

The following is \cite[Thm.\ 4.4]{BW20} (see also \cite[Thm.\ 3.17]{BDW20}). 
\begin{theorem}
	If $\tilde\pi$ has non-$Q$-critical slope, then it is non-$Q$-critical.
\end{theorem}

If $\tilde\pi$ is non-$Q$-critical, then there exists a unique Hecke eigenclass $\Phi_{\tilde\pi} \in \hc{t}(S_K,\sD_\lambda)$ with $r_\lambda(\Phi_{\tilde\pi}) = \phi_{\tilde\pi}$. If $\beta \in \Z_{\geq 1}^{\pri|p}$, then let $(\alpha_p^\circ)^\beta = \prod_{\pri|p}(\alpha_{\pri}^\circ)^{\beta_{\pri}}$.

\begin{definition}\label{def:p-adic L-function}
	Let $\beta\in \Z_{\geq 1}^{\pri|p}$.  The \emph{$p$-adic $L$-function attached to $\tilde\pi$ and $i_p$} is 
	\[
		L_p^{i_p}(\tilde\pi) \defeq (\alpha_{p}^\circ)^{-\beta}\sEv_\beta^\dagger(\Phi_{\tilde\pi}) \in \cD(\Galp,L).
	\]
	By Proposition \ref{prop:evaluations changing beta} (with $\bullet = *$), this is independent of $\beta$. Dependence on $i_p$ is \eqref{eq:phi} (via \eqref{eq:Theta p}).
\end{definition}

\subsection{Growth conditions}
The following is an adaptation of \cite[Def.\ 2.14]{Loe14} for our setting. For $\bm = (m_{\pri}) \in \Z_{\geq 1}^{\pri|p}$, let $p^{\bm} = \prod \pri^{m_{\pri}}$. We have an exact sequence
\[
	0 \to  \sU_{\bm} \to \Galp \to \cl_F^+(p^{\bm}) \to 0,
\]
where $\sU_{\bm}$ is the image of $1+p^{\bm}\cO_F\otimes\Zp$ under the natural map $(\cO_F\otimes\Zp)^\times \to \Galp$ from \eqref{eq:cl SES}. Let 
\[
	\cA_{\bm}(\Galp,L) \defeq \{f \in \cA(\Galp,L) : f\text{ is analytic on every translate of $\sU_{\bm}$}\}.
\]

Then $\cA(\Galp,L) = \varinjlim_{\bm} \cA_{\bm}(\Galp,L)$ is the union of all the $\cA_{\bm}$'s. Moreover each $\cA_{\bm}$ is a Banach $L$-space with respect to a discretely valued norm $||\cdot||_{\bm}$. Dualising gives a family of norms 
\begin{align}
	||\mu||_{\bm} &\defeq \mathrm{sup}_{f\in\cA_{\bm}(\Galp,L)}\tfrac{|\mu(f)|}{||f||_{\bm}}\notag\\
	&= \mathrm{sup}_{||f||_{\bm} \leqslant 1}|\mu(f)|\label{eq:sup norm}
\end{align}
on $\cD(\Galp,L)$, which thus obtains the structure of a Fr\'{e}chet module.

\begin{definition} \label{def:admissible}
	Let $\bh = (h_{\pri}) \in \Q_{\geqslant 0}^{\pri|p}$. We say $\mu \in \cD(\Galp,L)$ is \emph{admissible of growth $\bh$} if there exists $C \geqslant 0$ such that for each $\bm \in \Z_{\geqslant 1}^{\pri|p},$ we have $||\mu||_{\bm} \leqslant p^{\bh\bm }C,$ where $\bh\bm = (h_{\pri}m_{\pri})_{\pri}$.
\end{definition}

Note that if $\mu$ is admissible of growth $(0)_{\pri|p}$, then it defines a bounded measure. Indeed, if $X \subset \Galp$ is open compact, define its volume to be $\mu(1_X)$, where $1_X$ is the indicator function. For $\bm$ such that $1_X \in \cA_{\bm}(\Galp,L)$, we see $|\mu(1_X)| = \tfrac{|\mu(1_X)|}{||1_X||_{\bm}} \leq ||\mu||_{\bm} \leq C$, so $\mu$ is bounded.

\begin{proposition}\label{prop:admissible}
	Let $\Phi \in \hc{t}(S_K,\sD_\lambda)$ be a Hecke eigenclass with $U_{\pri}^* \Phi = \alpha_{\pri}^\circ\Phi$ for all $\pri$.  The distribution $\sEv_\beta^\dagger(\Phi) \in \cD(\Galp,L)$ is admissible of growth $\bh_p \defeq (v_p(\alpha_{\pri}^\circ))_{\pri|p}$.
\end{proposition}
\begin{proof}
	This is a higher-variable version of \cite[Prop.\ 6.20]{BDW20} and \cite[Prop.\ 3.11]{BDJ17}, and it is proved essentially identically. We indicate changes of notation. One defines $\beta_{\bm} = (m_{\pri}e_{\pri})_{\pri|p}$, and replaces all instances of $m$  \emph{op.\ cit}. (a positive integer) with $\bm$. After (6.21) \emph{op.\ cit}., $g$ should instead be analytic on $N_Q^\beta(\Zp)$, and $\xi$ there replaced with $u^{-1}$ here. The space $\Galp[\mathrm{pr}_{\beta_m}([\delta])]$ there is the translate of $\sU_{\bm}$ by $\delta$ here. The rest of the proof is identical.
\end{proof}

\subsection{Proof of Theorem \ref{thm:intro}}\label{sec:proof theorem A}
Finally we assemble our results and prove Theorem \ref{thm:intro} from the introduction. We have constructed the distribution $L_p^{i_p}(\tilde\pi)$ in Definition \ref{def:p-adic L-function}. It has the required growth property by Proposition \ref{prop:admissible}. For the interpolation, if $\chi$ has conductor dividing $p^\beta$, we have
\[
	\int_{\Galp} \chi_{[p]} \cdot dL_p^{i_p}(\tilde\pi) = (\alpha_p^\circ)^{-\beta}\int_{\Galp} \chi_{[p]} \cdot d\sEv_\beta^\dagger(\Phi_{\tilde\pi}) = (\alpha_p^\circ)^{-\beta} \cdot \sEv_{\chi}^{[*]}(r_\lambda(\Phi_{\tilde\pi})),
\]
by Proposition \ref{prop:evaluations interpolate}. Since $r_\lambda(\Phi_{\tilde\pi}) = \phi_{\tilde\pi} = \Theta_{[\omega]}^{K,i_p}(\varphi_{f}^{\mathrm{FJ},\alpha})$, this gives $(\alpha_p^\circ)^{-\beta}$ times the right-hand side of Theorem \ref{thm:critical value}, which has a factor $\lambda(t_p^{\beta})$. By definition of $\alpha_{\pri}^\circ$, we have 
	\[
		\lambda(t_p^\beta)(\alpha_p^\circ)^{-\beta} = \alpha_p^{-\beta} = \prod_{\pri|p} \alpha_{\pri}^{-\beta_\pri}.
	\]
	 If $\chi_{\pri}$ is unramified, then $\beta_{\pri} = 1$, and this factor of $\alpha_{\pri}$ cancels the one appearing in $Q'(\pi_{\pri},\chi_{\pri})$ in Theorem \ref{thm:critical value}, leaving exactly $e(\pi_{\pri},\chi_{\pri})$ from Theorem \ref{thm:intro}. If $\chi_{\pri}$ is ramified, then this factor of $\alpha_{\pri}^{-\beta}$ is exactly the difference between $Q'(\pi_{\pri},\chi_{\pri})$ and $e(\pi_{\pri},\chi_{\pri})$. 
	 
	 This gives exactly the interpolation formula of Theorem \ref{thm:intro}, except for one additional term: the sign $\chi_{[p]}(\det w_n)$. To remove this, we simply renormalise $L_p^{i_p}(\tilde\pi)$ by multiplying it by the Dirac measure $[\det w_n]$, that is, the bounded measure on $\Galp$ defined by $\int_{\Galp} \chi \cdot d[\det w_n] = \chi(\det w_n)$. This process removes the term $\chi_{[p]}(\det w_n)$, proving the claimed interpolation formula and completing the proof of Theorem \ref{thm:intro}. \qed

\subsection{Unicity in the imaginary quadratic case}\label{sec:IQF}

Now let $F$ be imaginary quadratic. The purity condition on $\lambda$ (Definition \ref{def:pure}) implies that $\lambda_{\sigma,n} - \lambda_{\sigma,n+1} = \lambda_{c\sigma,n} - \lambda_{c\sigma,n+1} = k$, say, so the range of standard critical infinity types for $\lambda$ forms the $(k+1)\times (k+1)$ square $\{(j_\sigma,j_{c\sigma}) \in \Z^2 : -\lambda_{\sigma,n+1} \geq j_\sigma \geq -\lambda_{\sigma,n}, -\lambda_{c\sigma,n+1} \geq j_{c\sigma} \geq -\lambda_{c\sigma,n}$. The non-$Q$-critical slope bound says, then, that `the growth is strictly less than the number of critical values'. Let $\tilde\pi$ be as in Theorem \ref{thm:intro}. The following is proved in \cite{Loe14} (cf.\ \cite[\S7.5]{Wil17}).

\begin{proposition}\label{prop:unique}
	If $\tilde\pi$ is non-$Q$-critical slope, then $L_p^{i_p}(\tilde\pi)$ is uniquely determined by its growth and interpolation properties.
\end{proposition}

When $F$ is totally real, the analogous unicity results are discussed in \cite[Prop.\ 6.25]{BDW20}. If $F$ is not imaginary quadratic or totally real, unicity remains mysterious; see e.g.\ \cite[\S13]{BW_CJM}.

\bigskip
\bigskip

\footnotesize

\textbf{Data availability statement:} No new data were created in the preparation of this paper.

\textbf{Conflict of interest statement:} No conflicts of interest arose in the preparation of this paper.

\textbf{Funding declaration:} This paper was supported by EPSRC Postdoctoral Fellowship EP/T001615/2.

	\renewcommand{\refname}{\normalsize References} 
\newcommand{\etalchar}[1]{$^{#1}$}

\end{document}